\documentclass[12pt]{amsart}
\usepackage{amscd,amsmath,amsthm,amssymb}
\usepackage{pgf,tikz, pifont}
\usetikzlibrary{arrows}

\numberwithin{equation}{section}
\def\frk{\frak}               % font for "Fraktur"

\def\Phi{{\frk n}}
\def\Phi{{\frk N}}
\def\opn#1#2{\def#1{\operatorname{#2}}} % to make operators
\opn\chara{char} \opn\length{\ell} \opn\pd{pd} \opn\rk{rk}
\opn\projdim{proj\,dim} \opn\injdim{inj\,dim} \opn\rank{rank}
\opn\depth{depth} \opn\grade{grade} \opn\height{height}
\opn\embdim{emb\,dim} \opn\codim{codim}

\opn\Tr{Tr} \opn\bigrank{big\,rank}
\opn\superheight{superheight}\opn\lcm{lcm}
\opn\trdeg{tr\,deg}%\emph{
\opn\reg{reg} \opn\lreg{lreg} \opn\ini{in} \opn\lpd{lpd}
\opn\size{size}\opn\bigsize{bigsize}
\opn\cosize{cosize}\opn\bigcosize{bigcosize}
\opn\sdepth{sdepth}\opn\sreg{sreg}
\opn\link{link}\opn\fdepth{fdepth}
\opn\index{index}
\opn\index{index}
\opn\indeg{indeg}
\opn\N{N}
\opn\SSC{SSC}
\opn\SC{SC}
\opn\lk{lk}
\opn\div{div} \opn\Div{Div} \opn\cl{cl} \opn\Cl{Cl}
\opn\Spec{Spec} \opn\Supp{Supp} \opn\supp{supp} \opn\Sing{Sing}
\opn\Ass{Ass} \opn\Min{Min}\opn\Mon{Mon} \opn\dstab{dstab} \opn\astab{astab}
\opn\Syz{Syz}
\opn\reg{reg}
\opn\Ann{Ann} \opn\Rad{Rad} \opn\Soc{Soc}
\opn\Im{Im} \opn\Ker{Ker} \opn\Coker{Coker} \opn\Am{Am}
\opn\Hom{Hom} \opn\Tor{Tor} \opn\Ext{Ext} \opn\End{End}\opn\Der{Der}
\opn\Aut{Aut} \opn\id{id}

\opn\nat{nat}
\opn\pff{pf}%   \pf exists already
\opn\Pf{Pf} \opn\GL{GL} \opn\SL{SL} \opn\mod{mod} \opn\ord{ord}
\opn\Gin{Gin} \opn\Hilb{Hilb}\opn\sort{sort}
\opn\initial{init}
\opn\ende{end}
\opn\height{height}
\opn\type{type}
\opn\aff{aff} \opn\con{conv} \opn\relint{relint} \opn\st{st}
\opn\lk{lk} \opn\cn{cn} \opn\core{core} \opn\vol{vol}
\opn\link{link} \opn\Link{Link}\opn\lex{lex}
\opn\gr{gr}

\def\pot#1#2{#1[\kern-0.28ex[#2]\kern-0.28ex]}

\opn\dirlim{\underrightarrow{\lim}}
\opn\inivlim{\underleftarrow{\lim}}
\def\Implies{\ifmmode\Longrightarrow \else
        \unskip${}\Longrightarrow{}$\ignorespaces\fi}
\def\implies{\ifmmode\Rightarrow \else
        \unskip${}\Rightarrow{}$\ignorespaces\fi}
\def\iff{\ifmmode\Longleftrightarrow \else
        \unskip${}\Longleftrightarrow{}$\ignorespaces\fi}

\let\:=\colon
\newtheorem{Theorem}{Theorem}[section]
 \newtheorem{Lemma}[Theorem]{Lemma}
 \newtheorem{Corollary}[Theorem]{Corollary}
 \newtheorem{Proposition}[Theorem]{Proposition}
 \newtheorem{Remark}[Theorem]{Remark}
 
 \newtheorem{Example}[Theorem]{Example}
 
 \newtheorem{Definition}[Theorem]{Definition}

\let\epsilon\varepsilon
\let\kappa=\varkappa
\def\qed{\ifhmode\textqed\fi
      \ifmmode\ifinner\quad\qedsymbol\else\dispqed\fi\fi}
\def\textqed{\unskip\nobreak\penalty50
       \hskip2em\hbox{}\nobreak\hfil\qedsymbol
       \parfillskip=0pt \finalhyphendemerits=0}
\def\dispqed{\rlap{\qquad\qedsymbol}}

\opn\dis{dis}
\def\pnt{{\raise0.5mm\hbox{\large\bf.}}}

\opn\Lex{Lex}

\begin{document}

\title{Geometric regularity of  powers of two-dimensional squarefree monomial ideals}
\author{ Dancheng Lu}

\address{Dancheng Lu, School  of Mathematical Sciences, Soochow University, 215006 Suzhou, P.R.China}
\email{ludancheng@suda.edu.cn}



\begin{abstract} Let $I$ be a two-dimensional squarefree monomial ideal of a polynomial ring $S$. We evaluate the geometric regularity, $a_i$-invariants of $S/I^n$ for $i\geq 2$. It turns out that they are all linear functions in $n$ from $n=2$. Also it
is shown that  $\mbox{g-reg}(S/I^n)=\reg(S/I^{(n)})$ for all $n\geq 1$.
\end{abstract}

\maketitle

\section*{Introduction}

Let $S:=K[x_1,\ldots,x_r]$ be the polynomial ring in variables $x_1,\ldots,x_r$ over a field $K$ and $\mathfrak{m}$ the maximal homogeneous ideal of $S$. Let $M$ be a finitely generated graded $S$-module. For each $0\leq i\leq \dim M$, the $a_i$-invariant of $M$ is defined by
$$a_i(M):=\max\{t: H_{\mathfrak{m}}^i(M)_t\neq 0\},$$
where  $H_{\mathfrak{m}}^i(M)$ is the $i$-th local cohomology module of $M$ with support in $\mathfrak{m}$, and we understand $\max \emptyset=-\infty$. The regularity  of $M$ is defined by $$\reg (M):=\max\{a_i(M)+i: 0\leq i\leq \dim M\}.$$

Let $I$ be a homogeneous ideal of the polynomial ring $S$. It was proved  that  $\reg (S/I^n)$ is a linear function in $n$  for $n\gg 0$, see \cite{CHT, K,TW}. In other words, there exists integers $d,e$ and $n_0$ such that $\reg  (S/I^n)=dn+e$ for all $n\geq n_0$. Based on this result, many authors have studied properties of the regularity of powers of homogeneous ideals. Roughly speaking, their researches fall into two classes. The one  is devoted to understanding  the nature of integers  $d,e$ and $n_0$ for  some special or general  ideals $I$, see e.g. \cite{Ch,EU,HT}. The other  is to computing explicitly or to bounding the regularity function $\reg(S/I^n)$ for some special classes of ideals $I$, see e.g. \cite{BBH,BHT,JNS}.

Let $\Delta$ be a simplicial complex on $[r]:=\{1,2,\ldots,r\}$. The Stanley-Reisner ideal of $\Delta$ is defined to be the ideal  of $S$ $$I_{\Delta}:=(\mathbf{x}_F: F \mbox{ is a minimal non-face of  }  \Delta),$$ where $\mathbf{x}_F$ is the squarefree monomial $\prod_{i\in F}x_i$.
  Every two-dimensional squarefree monomial ideal containing no variables is the Stanley-Reisner ideal of a simplicial complex of dimension one. Note that a simplicial complex of dimension one  can be regarded as a simple graph that may contains  isolated vertices.   Two-dimensional squarefree monomial ideals attract many authors' interests.
For example, the  Buchsbaum property  of  symbolic powers and ordinary powers  of these ideals  was  studied in \cite{MN1} and \cite{MN2}, respectively, and  the Cohen-Macaulayness of  symbolic powers and ordinary powers of such ideals was  characterized in terms of the properties of their associated graphs in \cite{MT}. Recently, the regularity of symbolic powers of such ideals was computed explicitly in \cite{HT2}.

 Inspired by the  regularity for sheaves on projective spaces,  M.E.~Rossi et al. introduced the following weaker but natural notion of regularity  in \cite{RTV}.
\begin{Definition} \em Let $M$ be a finitely generated graded $S$-module. The {\it geometric regularity} of $M$ is defined by $$\mbox{g-reg}(M):=\max\{a_i(M)+i: i>0\}.$$
\end{Definition}

Let $I$ be an arbitrary two-dimensional squarefree monomial ideal  of $S$.
In this paper, we  evaluate   the geometric regularity of $S/I^n$,  see Theorem~\ref{4.1},  and then  obtain  the equality $\mbox{g-reg}(S/I^n)=\reg(S/I^{(n)})$ for all $n\geq 1$.

 The paper is structured as follows. In Section 1, we  recall some concepts and results which we need in this paper. In Section 2, we consider
 the   question: if $I$ is a monomial ideal in $S$ and $J:=(I, yx_1, \cdots, yx_r)$ is the ideal of $R:=S[y]$, could we compare the regularity or $a_i$-invariants of $S/I^n$ with the ones of $R/J^n$? This question is inspired by the following observation: if $G$ is a simple graph on $[r]$ and $G'$ is the graph obtaining from $G$ by adding an isolated vertex $r+1$, then $I_{G'}=(I_G, x_{r+1}x_1,x_{r+1}x_2,\ldots,x_{r+1}x_r )$. With $S,I,R$ and $J$ defined as above, we prove  among other things that $a_1(R/J^n)=\max\{2n-2, a_1(S/I^{n-t})+t: 0\leq t\leq n-1\}$ if $\sqrt I\neq (x_1,\ldots,x_n)$, and  $a_i(R/J^n)=\max\{a_i(S/I^{n-t})+t: 0\leq t\leq n-1\}$ for $i\geq 2$, see Theorem~\ref{2.12}.

In Section 3, we compute   $a_i(S/I_G^n)$ for $i=1,2$ when $G$ is a simple graph without any isolated vertices. We find that $a_2(S/I_G^n)=a_2(S/I_G^{(n)})$ holds for all such graphs $G$ but $a_1(S/I_G^n)$ and $a_1(S/I_G^{(n)})$ may be very different.  Since $a_1(S/I_G^n)$ has been computed we can clarify all graphs $G$ and all $n>0$ for which $S/I_G^n$ is Cohen-Macaulay. This recovers two results of \cite{MT}. In the final section, by applying the afore-mentioned result obtained in Section 2, we can get  the values of  $\mbox{g-reg} (S/I_G^n)$ for any graph $G$  which may contain isolated vertices and for all $n\geq 1$. We conclude this paper by showing  $$\mbox{g-reg}(S/I^n)=\reg(S/I^{(n)})$$ for all two-dimensional squarefree monomial ideals $I$.

\section{Preliminaries}

In this section, we fix notation and recall some concepts and results which will be used in this paper. Throughout this paper,    we let $[r]:=\{1,2,\ldots,r\}$ and $S:=K[x_1,\ldots,x_r]$, the polynomial ring in variables $x_1,\ldots,x_r$ over a field $K$.

\subsection{Geometric Regularity}  We refer to \cite{BS} for the knowledge of local cohomology. It is known that $H_{\mathfrak{m}}^0(M/H_{\mathfrak{m}}^0(M))=0$ and $H_{\mathfrak{m}}^i(M)=H_{\mathfrak{m}}^i(M/H_{\mathfrak{m}}^0(M))$ for all $i>0$, see e.g. \cite[Chapter 2]{BS}. In particular, we have
$$\mbox{g-reg}(M)=\mathrm{reg}(M/H_{\mathfrak{m}}^0(M)).$$

Let $I,J$ be graded ideals of $S$. We set as usual $$I:J:=\{x\in S: xJ\subseteq I\}$$
 and  $$I:J^{\infty}:=\bigcup_{i\geq 1}I:J^i.$$
The ideal $I:\mathfrak{m}^{\infty}$ is called the {\it saturation} of $I$.
 Note that $H_{\mathfrak{m}}^0(S/I)=(I:\mathfrak{m}^{\infty})/I$,  we obtain  $$\mbox{g-reg}(S/I)=\mathrm{reg}(S/I:\mathfrak{m}^{\infty}).$$

\subsection{Simplicial Complex} Recall from \cite{HH} that a {\it simplicial complex} $\Delta$ on $[r]$ is a collection of subsets of $[r]$ such that $\{i\}\in \Delta$  for each $i\in [r]$ and that if $\sigma\in \Delta$ and $\tau\subseteq \sigma$ then $\tau\in \Delta$. The elements  $F\in \Delta$ are called  {\it faces} of $\Delta$, and the {\it dimension} of each face $F\in \Delta$ is defined by $\dim F=|F|-1$, where $|F|$ is the  cardinality of $A$.  Also the {\it dimension} of $\Delta$, $\dim \Delta$, is given by $\max\{\dim F: F\in \Delta\}$. Hence, a simplicial complex of dimension one is a simple graph that may contain some isolated vertices. A {\em facet} is a maximal face of $\Delta$ (with respect to inclusion). Let $\mathcal{F}(\Delta)$ denote the set of $\Delta$. It is clear that $\mathcal{F}(\Delta)$ governs $\Delta$. When $\mathcal{F}(\Delta)=\{F_1,\cdots,F_k\}$ , we write $\Delta=\langle F_1,\ldots,F_k\rangle$.

A {\it  non-face} of $\Delta$ is a subset $F$ of $[r]$ with $F\notin \Delta$ and let $\mathcal{N}(\Delta)$ denote the set of minimal non-face of $\Delta$. For any subset $F$ of $[r]$, we set $$\mathbf{x}_{F}:=\prod_{i\in F}x_i.$$
The Stanley-Reisner ideal of $\Delta$ is the ideal $I_{\Delta}$ which is generated by $\mathbf{x}_{F}$ with $F\notin \Delta$. In other words, $$I_{\Delta}:=(x_F: F\in \mathcal{N}(\Delta)).$$   By \cite[Lemma 1.5.4]{HH}, $I_{\Delta}$ has the following primary decomposition
$$I_{\Delta}=\bigcap_{F\in \mathcal{F}(\Delta)} P_{\overline{F}},$$ where $P_{\overline{F}}:=(x_i: i\in [r]\setminus F).$ From this decomposition, we see that $$\mathrm{Krull-dim} (S/I_{\Delta})=\dim \Delta+1.$$   Thus, if $G$ is a simple graph, then the Stanley-Reisner ideal $I_G$ is a two-dimensional squarefree monomial ideal of $S$. Conversely, any two-dimensional squarefree monomial ideal of $S$  containing no variables arises in this fashion.

Recall from \cite{HH} that the {\it augmented oriented chain complex} $\widetilde{\mathcal{C}}(\Delta; K)$ of $\Delta$ with respect to $K$ is defined as follows. A $K-$basis of $\widetilde{\mathcal{C}}_j(\Delta; K)$ is given by $\mathbf{e}_F$ with $\dim F=j$ and $F\in \Delta$.  For $F=\{i_0<i_1<\cdots<i_j\}$ one denotes the element $\mathbf{e}_F$ by $[i_0,i_1,\ldots,i_j]$.  With this notation, the chain map $\partial: \widetilde{\mathcal{C}}_j(\Delta; K)\rightarrow \widetilde{\mathcal{C}}_{j-1} (\Delta; K)$ is given by $$\partial([i_0,i_1,\ldots,i_j])=\sum_{k=0}^j(-1)^k[i_0,\ldots, i_{k-1},i_{k+1},\ldots,i_j].$$
The simplicial homology  $\widetilde{H}_j(\Delta;K)$ is then  defined as the $j$-th homology group of the complex $\widetilde{C}(\Delta;K)$.  That is,  $$\widetilde{H}_j(\Delta;K):=H_j(\widetilde{\mathcal{C}}(\Delta; K)).$$
We collect some easy facts on simplicial homology we need in the following lemmas. Recall that a simplicial complex is a {\it simplex} if it has a unique facet.
\begin{Lemma} \label{sh1} {\em (1)} $\widetilde{H}_{-1}(\Delta;K)\neq 0$ if and only if $\Delta=\{\emptyset\}$.

{\em (2)} If $\Delta$ is a simplex, then $\widetilde{H}_i(\Delta;K)=\widetilde{H}_i(\emptyset;K)=0$ for all $i\in \mathbb{Z}$.

{\em (3)} $\widetilde{H}_0(\Delta;K)\neq 0$ if and only if $\Delta$ is disconnected.

{\em (4)} Let $\Delta$ be a simplicial complex on $[r]$ such that  $\Delta\neq \emptyset$ and  $\Delta\neq \{\emptyset\}$. Then $\widetilde{H}_i(\Delta_1;K)=\widetilde{H}_i(\Delta;K)$ for $i\neq 0$ and $\dim_K\widetilde{H}_0(\Delta_1;K)=\dim_K\widetilde{H}_0(\Delta;K)+1$, where $\Delta_1:=\Delta\cup \{r+1\}$ is the simplicial complex on $[r+1]$.

\end{Lemma}
\begin{proof} The statements (1), (3) and (4) are immediate from the definition. For (2), one may see e.g. \cite[Example 5.1.9]{HH}.
\end{proof}

The following is a copy of  \cite[Lemma 1.6]{HT2}.

\begin{Lemma} \label{sh2}

Let $G$ be a simple graph considered as a simplicial complex of dimension one. Then $\widetilde{H}_1(G;K)=0$ if and only if $G$ contains no cycles.
\end{Lemma}

\subsection{Takayama's Formula}

Let $I$ be a monomial ideal of the polynomial ring $S$ and $\mathfrak{m}$ the maximal homogeneous ideal of $S$. For any $\mathbf{a}=(a_1,\ldots,a_r)\in \mathbb{Z}^r$, we  put $G_{\mathbf{a}}=\{i\in [r]:a_i<0\}$.  Let $\Delta(I)$ denote the simplicial complex of all $F\subseteq [r]$ such that $x_F\notin \sqrt I$. The famous Takayama's formula  \cite{T} can be stated as follows.

 \begin{Lemma} \label{Takayama} $\dim_K  H^{i}_{\mathfrak{m}}(S/I)_{\mathbf{a}}$=$\left\{
                                                                        \begin{array}{ll}
                                                                         \dim_K \widetilde{H}_{i-|G_{\mathbf{a}}|-1}(\Delta_{\mathbf{a}}(I);K) , & \hbox{$G_{\mathbf{a}}\in \Delta(I)$;} \\
                                                                          0, & \hbox{otherwise.}
                                                                        \end{array}
                                                                      \right.$

 \end{Lemma}

 The simplicial complex $\Delta_{\mathbf{a}}(I)$ has several equivalent interpretations.
Suppose that $I$ is generated by monomials $(u_1,\ldots,u_g)$. Then
\begin{equation*}
\begin{aligned}
 \Delta_{\mathbf{a}}(I)  & =\{F\subseteq [r]\setminus G_{\mathbf{a}}:  \mathbf{x}^{\mathbf{a}}\notin IS_{F\cup G_{\mathbf{a}}}\}\\
&=\{F\subseteq [r]\setminus G_{\mathbf{a}}: \forall 1\leq j\leq g, \exists i\in [r]\setminus (F\cup G_{\mathbf{a}}) \mbox{ with } a_i<\deg_i (u_j)\}\\
&=\{F\subseteq [r]\setminus G_{\mathbf{a}}: \mathbf{x}^{\mathbf{a}}\mathbf{x}_{F\cup G_{\mathbf{a}}}^t\notin I, \forall t\geq 1\}.
\end{aligned}
\end{equation*}
Here, $S_{F\cup G_{\mathbf{a}}}:=S[x_i^{-1}:i\in F\cup G_{\mathbf{a}}]$, and $\deg_i(u)$ is defined to be $a_i$ if $u=x_1^{a_1}\cdots x_r^{a_r}$.

The concept of {\em  monomial localization} is introduced in \cite{HRV} as a simplification of the localization. Fix a subset $F\subseteq [r]$. Let $\pi_F: S\rightarrow K[x_i: i\in [r]\setminus F]$ be  the  $K$-algebra
 homomorphism extended by the map sending $x_i$ to $x_i$ for $i\in [r]\setminus F$ and $x_i$ to 1 for $i\in F$. The image of a monomial ideal $I$ of $S$ under the map $\pi_F$ is called the monomial localization of $I$ with respect to $F$, denoted by $I[F]$.
It is clear that if $I,J$ are monomial ideals of $S$, then $(IJ)[F]=I[F]J[F]$ and $(I\cap J)[F]=I[F]\cap J[F]$. Let $\mathbf{a}_{+}$ denote the non-negative part of a vector $\mathbf{a}$.   Under these notations we have the following description of $\Delta_{\mathbf{a}}(I)$.

\begin{Lemma}\label{local} In Lemma~\ref{Takayama},
 $\Delta_{\mathbf{a}}(I)=\{F\subseteq [r]\setminus G_{\mathbf{a}}:  \mathbf{x}^{\mathbf{a}_{+}}\notin I[F\cup G_{\mathbf{a}}]S\}.$
\end{Lemma}

\begin{proof} This is because $I[F\cup G_{\mathbf{a}}]S=IS_{F\cup G_{\mathbf{a}}}\cap S$.\end{proof}

 We will use the above interpretation of $\Delta_{\mathbf{a}}(I)$ in this paper. Let $G\subseteq [r]$. Recall that $\mathrm{Link}_{\Delta}(G)$ is defined to be the subcomplex  $\{F\setminus G:G\subseteq F\in \Delta\}$.  The following result allows us to  consider only the case when $\mathbf{a}\in \mathbb{N}^r$.
\begin{Lemma} \label{transfer} Let  $\mathbf{a}=(a_1,\ldots, a_r)$ be a vector in $\mathbb{Z}^r$ with $G_{\mathbf{a}}\neq \emptyset$. Then $$\Delta_{\mathbf{a}}(I)=\mathrm{Link}_{\Delta_{\mathbf{a_+}}(I)}(G_{\mathbf{a}}).$$
\end{Lemma}
\begin{proof}
If $F$ belongs to either $\Delta_{\mathbf{a}}(I)$ or $\mathrm{Link}_{\Delta_{\mathbf{a_+}}(I)}(G_{\mathbf{a}})$ then $F\subseteq [r]\setminus G_{\mathbf{a}}$. Now let $ F$ be a subset of $[r]\setminus G_{\mathbf{a}}$. Then
$F\in \Delta_{\mathbf{a}}(I)$ if and only if $\mathbf{x}^{\mathbf{a+}}\notin I[F\cup G_{\mathbf{a}}]S$ if and only if  $F\cup G_{\mathbf{a}}\in \Delta_{\mathbf{a}_+}(I)$ if and only if $F\in \mathrm{link}_{\Delta_{\mathbf{a}_+}}(G_{\mathbf{a}})$, as desired. \end{proof}

We close this section by giving a simplification of Takayama's formula in the case that $I$ is a power of a squarefree monomial ideal.

\begin{Lemma} \label{lu} Let $I$ be a squarefree monomial ideal of $S$. Then, for all $\mathbf{a}\in \mathbb{Z}^r$ and $n\geq 1$, we have   $\dim_K  H^{i}_{\mathfrak{m}}(S/I^n)_{\mathbf{a}}=\dim_K \widetilde{H}_{i-|G_{\mathbf{a}}|-1}(\Delta_{\mathbf{a}}(I^n);K).$
\end{Lemma}

\begin{proof} In view of Lemma~\ref{Takayama} as well as Lemma~\ref{sh1}.(2), it is enough to show that if $G_{\mathbf{a}}\notin \Delta(I^n)$ then $\Delta_{\mathbf{a}}(I^n)=\emptyset$.

 Assume now that $G_{\mathbf{a}}\notin \Delta(I^n)$. Then $\mathbf{x}_{G_{\mathbf{a}}}\in \sqrt {I^n}=I$ and so $1\in I[G_{\mathbf{a}}]$. From this  it follows that $1\in I^n[G_{\mathbf{a}}]$, and thus $\Delta_{\mathbf{a}}(I^n)=\emptyset$ by Lemma~\ref{local}, as required.
\end{proof}

\section{Some Comparisons}

In this section we always assume that $I$ is  a monomial ideal of $S$ with $I\neq S$ and  let $$J:=(I,x_1y,\ldots,x_ry)$$ be the ideal of $R:=S[y]=K[x_1,\ldots,x_r,y]$.  We fix $n\geq 1$ and  compare the regularity, g-regularity and   $a_i$-invariants of $S/I^n$ with the ones of $R/J^k$.

We begin with basic lemmas on the regularity of graded modules.

\begin{Lemma} \label{2.1} Let $0\rightarrow M\rightarrow N\rightarrow P\rightarrow 0$ be a short exact sequence of finitely generated graded $S$-modules. Then

{\em (1)} $\reg(N)\leq \max\{\reg(M),\reg(P)\}$;

{\em(2)} $\reg(M)\leq \max\{\reg(N),\reg(P)+1\}$;

{\em (3)} $\reg(P)\leq \max\{\reg(N),\reg(M)-1\}$;

{\em (4)}  $\reg(N)=\reg(P)$ if $\reg(M)\leq \reg(P).$
\end{Lemma}

\begin{proof} The first three statements are well-known, see e.g. \cite{CH}.
           For the convenience of the readers, we give a proof of (4). Since $\reg(M)\leq \reg(P)$, we  obtain  $\reg(P)\leq \reg(N)$  by (3), and    $\reg(N)\leq \reg(P)$ by (1). Thus  $\reg(N)=\reg(P)$, as desired. \end{proof}

 Denote by  $\mathfrak{m}$ the maximal homogenous ideal $(x_1,\ldots,x_r)$ of $S$. The following lemma is   a special case of \cite[Theorem 2.2]{CH}, but we include a short proof for the completeness.

\begin{Lemma} \label{2.2} Let $M$ be a finitely generated graded $S$-module. Then $\reg (\mathfrak{m}M)\leq \reg (M)+1$.
\end{Lemma}
\begin{proof}  Since $H_{\mathfrak{m}}^0(M/\mathfrak{m}M)=M/\mathfrak{m}M$,  we have $\reg (M/\mathfrak{m}M)=a_0(M/\mathfrak{m}M)$ and it is the largest degree of minimal generators of $M$, which, by e.g. \cite[Theorem 15.3.1]{BS}, is less than or equal to $\reg(M)$. Thus, the desired  inequality follows from  the short exact sequence $0\rightarrow \mathfrak{m}M\rightarrow M\rightarrow M/\mathfrak{m}M\rightarrow 0$ together with Lemma~\ref{2.1}.(2).
 \end{proof}

\begin{Proposition} \label{2.3}  Assume    $I\subseteq \mathfrak{m}^2$. Then the following statements hold.

{\em (1)} If  $\reg(S/I^n)=dn+e$ $n\gg 0$ with $d\geq 3$,  then $\reg (R/J^n)=\reg (S/I^n)$ for all $n\gg 0$;

{\em (2)} Given any positive integer $\lambda$, if  $I^n$ has a linear resolution for all $n\leq \lambda$,  then $\reg(R/J^n)=\reg (S/I^n)$ for all  $n\leq \lambda$.

\end{Proposition}

\begin{proof}
We will use the following short exact sequence:
\begin{gather} \label{exact1}\tag{$\dag$}
 0\rightarrow (R/J^n:y)[-1]\rightarrow R/J^n\rightarrow R/(J^n,y)\rightarrow 0.
\end{gather}
For this, we  first look at the regularities of  $(R/J^n:y)[-1]$ and $R/(J^n,y)$.

  Note that $R=\bigoplus_{i\geq 0} S[y^i]$ and $J=(I,\mathfrak{m}y)$, it is not difficult to see $$J^n=I^n\oplus I^{n-1}\mathfrak{m}y\oplus \cdots\oplus I{\mathfrak{m}}^{n-1}y^{n-1}\oplus \oplus_{i\geq 0}{\mathfrak{m}}^ny^{n+i}.$$
      This implies $$J^n:y=I^{n-1}\mathfrak{m}\oplus I^{n-2}\mathfrak{m}^2y\oplus \cdots\oplus I\mathfrak{m}^{n-1}y^{n-2}\oplus \oplus_{i\geq 0}\mathfrak{m}^ny^{n-1+i}$$ and so $$R/J^n:y=S/I^{n-1}\mathfrak{m}\oplus (S/I^{n-2}\mathfrak{m}^2)y\oplus \cdots \oplus (S/I\mathfrak{m}^{n-1})y^{n-2} \oplus (S[y]/\mathfrak{m}^n)(y^{n-1}),$$ where the last equality follows from the equality $$(S/\mathfrak{m}^n)y^{n-1}\oplus (S/\mathfrak{m}^n)y^n\oplus \cdots=(S[y]/\mathfrak{m}^n)y^{n-1}.$$
From this it follows that
\begin{equation}\label{in}\tag{$\ddag$}
\begin{split}
 &\reg(R/J^n:y)[-1]=\reg(R/J^n:y)+1\\
& =\max\{\reg(S/I^{n-k}\mathfrak{m}^k)+k: 1\leq k\leq n-1,  \reg(S/\mathfrak{m}^{n})+n(=2n-1)\}.
\end{split}
\end{equation}
On the other hand, it is clear that $(J^n,y)=(I^n,y)$. Thus, $\reg R/(J^n,y)=\reg (S/I^n)$.

(1)  We may assume $\reg(S/I^n)=dn+e$ for all $n\geq n_0$.  Then, by Lemma~\ref{2.2},  one has $$\reg(S/I^{n-k}\mathfrak{m}^k)+k\leq\left\{
                                       \begin{array}{ll}
                                         d(n-k)+e+2k, & \hbox{if $n_0\leq n-k$;} \\
                                         t+2k, & \hbox{if $1\leq  n-k\leq n_0$.}
                                       \end{array}
                                     \right.,$$
  where $t:=\max\{\reg(S/I^n):n\leq n_0\}$. Let $n_0':=n_0+t$. Since $d\geq 3$, it follows that  $\reg(S/I^{n-k}\mathfrak{m}^k)+k\leq d(n-1)+e+2$ for all $1\leq k\leq n-1$ and for $n\geq n'_0$, and so $\reg(R/J^n:y)[-1]\leq d(n-1)+e+2$ for all $n\geq n_0'$ by (\ref{in}). This implies $$\reg(R/J^n:y)[-1]\leq \reg(R/(J^n,y))$$ for $n\gg 0.$
Now the afore-mentioned short exact sequence (\ref{exact1}) yields the desired equality in view of  Lemma~\ref{2.1}.(4).

(2) We may  assume that $I$ is generated in a single degree $d\geq 2$. Then $\reg(S/I^n)=dn-1$ and this implies $\reg(R/J^n:y)[-1]\leq d(n-1)+1$ for $n\leq \lambda$ by  (\ref{in}). From this it follows that $\reg(R/J^n:y)[-1]\leq \reg(R/(J^n,y))$ for $n\leq \lambda$, and then the desired equality follows by using the short exact sequence (\ref{exact1}) and by Lemma~\ref{2.1}.(4).
\end{proof}

\begin{Remark} {\em (1) From this proof, the condition that $I$ is a monomial ideal is not necessary in Proposition~\ref{2.3}. That is, $I$ could be any homogeneous ideal.

(2) \cite[Remark 5.7]{NV} shows that the condition that $d\geq 3$ in Proposition~\ref{2.3}.(1)  cannot be removed. But if assume further $I$ is a squarefree monomial ideal, then the condition that $d\geq 3$ could be dropped, as shown by \cite[Corollary 5.6]{NV}.  }

\end{Remark}
 Let  $\mathfrak{n}$ denote the maximal homogeneous ideal $(x_1,\ldots,x_r,y)$ of $R$.

\begin{Proposition} \label{2.4}  Suppose  $I\subseteq \mathfrak{m}^2$. Then {\em $\mbox{g-reg}(R/J^n)\geq n-1$ } for all $n\geq 1$.
\end{Proposition}

\begin{proof}  We claim that if $\mathbf{x}^{\mathbf{a}}y^t\in J^n: \mathfrak{n}^{\infty}$ for some $t\geq 0$, then $|\mathbf{a}|\geq n$.
Let $\mathbf{x}^{\mathbf{a}}y^t\in J^n: \mathfrak{n}^{\infty}$. Then there exists $k\geq 1$ such that $\mathbf{x}^{\mathbf{a}}y^{t+k}\in J^n$. It follows that $$\mathbf{x}^{\mathbf{a}}y^{t+k}=u_{i_1}u_{i_2}\cdots u_{i_n}v,$$ where $u_{i_j}$ are minimal generators of $J$ for $j=1,\ldots,n$ and $v$ is a monomial. Since $\deg_{\mathbf{x}}(u_{i_j})\geq 1$ for all $j$, we have  $|\mathbf{a}|=\deg_{\mathbf{x}}(\mathbf{x}^{\mathbf{a}}y^{t+k})\geq \deg_{\mathbf{x}}(u_{i_1})+\cdots+\deg_{\mathbf{x}}(u_{i_n})\geq n$, as claimed. Here  $\deg_{\mathbf{x}} (u)$ is defined to be the integer $|\mathbf{a}|$ if $u=\mathbf{x}^{\mathbf{a}}y^t$.

From this claim it follows that   $$\mbox{g-reg}(R/J^n)=\reg(R/J^n:\mathfrak{n}^{\infty})=\reg (J^n:\mathfrak{n}^{\infty})-1\geq n-1,$$ as desired. Here, the last inequality follows from \cite[Theorem 15.3.1]{BS}.
\end{proof}
\begin{Example} \label{Example} Let $I$ be the ideal $\mathfrak{m}^2$. Then

{\em (1)} $\reg(R/J^n)=\reg(S/I^n)=2n-1$;

{\em (2)} {\em $\mbox{g-reg}(R/J^n)=n-1$ and $\mbox{g-reg}(S/I^n)=-\infty$.}
\end{Example}
\begin{proof}
(1) It  follows from Proposition~\ref{2.3}.(2).

(2) We first show that $J^n:\mathfrak{n}^{\infty}=(x_1,\ldots,x_r)^n$. The inclusion has been proved in the proof of Proposition~\ref{2.4}. Let $u=\mathbf{x}^{\mathbf{a}}$ with $|\mathbf{a}|\geq n$. Then, for any $\mathbf{x}^{\mathbf{b}}y^t$ with $|\mathbf{b}|+t= n$, we may choose a vector $\mathbf{c}\in \mathbb{N}^r$ such that $|\mathbf{c}|=t$ and $\mathbf{c}\leq \mathbf{a+b}$. Since $\mathbf{x}^{\mathbf{a+b-c}}\in (x_1,\ldots,x_r)^{2n-2t}=I^{n-t}$ and $\mathbf{x^c}y^t\in (x_1y,\cdots,x_ry)^t$, it follows that  $u\mathbf{x}^{\mathbf{b}}y^t\in J^n$, and so $u\in J^n:\mathfrak{n}^{\infty}$, by noting that $\mathfrak{n}^n$ is generated by $\mathbf{x}^{\mathbf{b}}y^t$ with $|\mathbf{b}|+t=n$. Thus we have shown $J^n:\mathfrak{n}^{\infty}=(x_1,\ldots,x_r)^n$. Consequently, $$\mbox{g-reg}(R/J^n)=\reg(R/J^n:\mathfrak{n}^{\infty})=\reg (J^n:\mathfrak{n}^{\infty})-1= n-1,$$ for all $n\geq 1$.  Notice that  $I^n:\mathfrak{m}^{\infty}=S$, the last statement follows.
\end{proof}

This example suggests the relationship between $\mbox{g-reg}(R/J^n)$ and $\mbox{g-reg}(S/I^n)$ is more subtle. In the rest part of this section we  will express  the $a_i$-invariants   of $R/J^n$ in terms of the ones of $S/I^{n-t}$ with $0\leq t\leq n-1$ and for $i\geq 1$ by using Takayama's Lemma. For this, we need  some preparations   describing the relations between simplicial complexes $\Delta_{(\mathbf{a},t)}(J^n)$ and $\Delta_{\mathbf{a}}(I^n)$, where $\mathbf{a}$ is  a vector in $\mathbb{Z}^r$
and $t\in \mathbb{Z}$. Note that $(\mathbf{a},t)\in \mathbb{Z}^{r+1}$.

\begin{Proposition} \label{c2}

{\em (1)} Suppose  $t<0$. Then $\Delta_{(\mathbf{a},t)}(J^n)$ is either $\emptyset$ or $\{\emptyset\}$. Moreover, $\Delta_{(\mathbf{a},t)}(J^n)=\{\emptyset\}$ if and only if $|\mathbf{a}|\leq n-1$ and $G_{\mathbf{a}}=\emptyset$.
\vspace{2mm}

 {\em (2)} Suppose that $0 \leq t\leq n-1$ and  $F\subseteq [r]$.  If either $G_{\mathbf{a}}\neq \emptyset$ or $F\neq \emptyset$,  then  $F\in \Delta_{(\mathbf{a},t)}(J^n)\Longleftrightarrow F\in \Delta_{\mathbf{a}}(I^{n-t}).$
\vspace{2mm}

 {\em (3)} Suppose that $0 \leq t\leq n-1$ and  $F\subseteq [r+1]$. If $r+1\in F$, then $F\in \Delta_{(\mathbf{a},t)}(J^n)$ if and only if  $F=\{r+1\}$, $G_{\mathbf{a}}=\emptyset$ and $|\mathbf{a}|\leq n-1$.

\vspace{2mm}
 {\em (4)} Suppose  $t\geq n$.  Then $\Delta_{(\mathbf{a},t)}(J^n)$ is either $\emptyset$ or  $\langle\{r+1\}\rangle$. Moreover, $\Delta_{(\mathbf{a},t)}(J^n)=\langle\{r+1\}\rangle$ if and only if $G_{\mathbf{a}}=\emptyset$ and $|\mathbf{a}|\leq n-1$.

\end{Proposition}

\begin{proof} (1) Note that $J[r+1]=(x_1,\ldots,x_r)$ and $G_{(\mathbf{a},t)}=G_{\mathbf{a}}\cup \{r+1\}$. If $G_{\mathbf{a}}\neq \emptyset$, then $ J^n[G_{(\mathbf{a},t})]R=R$ and so $\Delta_{(\mathbf{a},t)}(J^n)=\emptyset$. Now assume  $G_{\mathbf{a}}= \emptyset$. Then, for any  $\emptyset \neq F\subseteq [r]$, we have $1\in J[F,G_{(\mathbf{a},t)}]$, and so $F\notin \Delta_{(\mathbf{a},t)}(J^n)$. This implies $\Delta_{(\mathbf{a},t)}(J^n)$ is either $\emptyset$ or $\{\emptyset\}$. Moreover, $\emptyset\in \Delta_{(\mathbf{a},t)}(J^n)$ if and only if $\mathbf{x}^{\mathbf{a}} \notin (x_1,\ldots,x_r)^n$ if and only if $|\mathbf{a}|\leq n-1$. This proves (1).

\vspace{2mm}

 (2)  If either $G_{\mathbf{a}}\neq \emptyset$ or $F\neq \emptyset$, then $J[F\cup G_{\mathbf{a}}]=(I[F\cup G_{\mathbf{a}}], y)$. Thus $\mathbf{x}^{\mathbf{a}}y^t\notin (J[F\cup G_{\mathbf{a}}])^n$ if and only if $\mathbf{x}^{\mathbf{a}}\notin (I[F\cup G_{\mathbf{a}}])^{n-t}$.
This proves (2). \vspace{2mm}

(3) If either $F\supsetneq \{r+1\}$ or $G_{\mathbf{a}}\neq \emptyset$, then  $J^n[F\cup G_{\mathbf{a}}]R=R$ and  $F\notin \Delta_{(\mathbf{a},t)}(J^n)$. Hence, if $F\in \Delta_{(\mathbf{a},t)}(J^n)$, we must have  $F=\{r+1\}$, $G_{\mathbf{a}}=\emptyset$.  Now assume that $F=\{r+1\}$ and $G_{\mathbf{a}}=\emptyset$. Then $J[G_{(\mathbf{a},t)}\cup F]=(x_1,\ldots,x_r)$.   Since $\mathbf{x}^{\mathbf{a}}\notin (x_1,\ldots,x_r)^n$ if and only if $|\mathbf{a}|\leq n-1$, the statement  (3) has been proved.
\vspace{2mm}

(4) Suppose  either $F\cap [r]\neq \emptyset$ or $G_{\mathbf{a}}\neq \emptyset$. Since $y\in  J[F\cup G_{\mathbf{a}}]$ and  $t\geq n$, we have $y^t\in (J[F\cup G_{\mathbf{a}}])^n$ and  $\mathbf{x^{a_+}}y^t\in (J[F\cup G_{\mathbf{a}}])^n$.  Thus, if $G_{\mathbf{a}}\neq \emptyset$ then $\Delta_{(\mathbf{a},t)}(J^n)=\emptyset$,  and if $G_{\mathbf{a}}=\emptyset$ and $F\in \Delta_{(\mathbf{a},t)}(J^n)$, then $F\cap [r]=\emptyset$.

Assume that $G_{\mathbf{a}}=\emptyset$ and $\Delta_{(\mathbf{a},t)}(J^n)\neq \emptyset$.  Then  $\emptyset\in \Delta_{(\mathbf{a},t)}(J^n)$ and it follows that  $\mathbf{x^{a}}y^t\notin J^n$. Since $t\geq n$, we have  $|\mathbf{a}|\leq n-1$, for otherwise, we have $\mathbf{x^{a}}y^t\subseteq (x_1y,\ldots,x_ry)^n\subseteq J^n$, a contradiction. From this it follows that $\mathbf{x^{a}}y^t\notin (J[r+1])^n=(x_1,\ldots,x_r)^n$ and so $\{r+1\}\in \Delta_{(\mathbf{a},t)}(J^n)$. Hence, we have  $\Delta_{(\mathbf{a},t)}(J^n)=\langle\{r+1\}\rangle$ and  the first statement has been proved. The second one follows immediately by the fact $\mathbf{x}^{\mathbf{a}}y^t\notin (J[r+1])^n$ if and only if $|\mathbf{a}|\leq n-1$.
\end{proof}

We now present the main result of this section.

\begin{Theorem}   \label{2.12}

{\em (1)}  If $i\geq 2$ then $a_i(R/J^n)=\max\{a_i(S/I^{n-t})+t: 0\leq t\leq n-1\}.$

{\em (2)} If $I[j]\neq S[j]$ for some $j\in [r]$, then $$a_1(R/J^n)=\max\{2n-2, a_1(S/I^{n-t})+t: 0\leq t\leq n-1\}.$$
\end{Theorem}
\begin{proof}
 Given any  integer $i\geq 1$, we first show that $a_i(R/J^n)\geq a_i(S/I^{n-t})+t$ for all $0\leq t\leq n-1$.
Fix $0\leq t\leq n-1$.  If  $a_i(S/I^{n-t})=-\infty$, there is nothing to prove.
Now, we assume $a_i(S/I^{n-t})\neq -\infty$ and let $\mathbf{a}\in \mathbb{Z}^r$ such that $H_{\mathfrak{m}}^i(S/I^{n-t})_{\mathbf{a}}\neq 0$ with $|\mathbf{a}|=a_i(S/I^{n-t})$. Since $\widetilde{H}_{i-|G_{\mathbf{a}}|-1}(\Delta_{\mathbf{a}}(I^{n-t});K)\neq 0$, see Lemma~\ref{Takayama}, one has $\Delta_{\mathbf{a}}(I^{n-t})\neq \emptyset$. There are two cases to consider.

We first consider the case that  $\Delta_{\mathbf{a}}(I^{n-t})=\{ \emptyset\}$. Since $\widetilde{H}_{i-|G_{\mathbf{a}}|-1}(\{\emptyset\};K)\neq 0$,
we have $|G_{\mathbf{a}}|=i\geq 1$ by Lemma~\ref{sh1}.(1). From this it follows that $r+1\notin F$ for any $F\in \Delta_{(\mathbf{a},t)}(J^n)$ by Proposition~\ref{c2}.(3) and so $\Delta_{(\mathbf{a},t)}(J^n)=\{\emptyset\}$ by  Proposition~\ref{c2}.(2). Note that $\sqrt J=(\sqrt I, x_1y,\ldots,x_ry)$, we have $\Delta (I^k)=\Delta (I)=\Delta(J) \cap 2^{[r]}=\Delta(J^\ell)$ for all $k,\ell \geq 1$. This implies $G_{(\mathbf{a},t)}=G_{\mathbf{a}}\in \Delta{(J^n)}$, and so  $H_{\mathfrak{n}}^i(R/J^n)_{(\mathbf{a},t)}=\widetilde{H}_{-1}(\{\emptyset\})\neq 0$ by Lemma~\ref{Takayama}. Thus, $a_i(R/J^n)\geq |\mathbf{a}|+t= a_i(S/I^{n-t})+t$.

 We next consider the case when $\Delta_{\mathbf{a}}(I^{n-t})\supsetneqq \{ \emptyset\}$. In this case we have  either $\Delta_{(\mathbf{a},t)}(J^n)=\Delta_{\mathbf{a}}(I^{n-t})\cup \{r+1\}$ and $G_{\mathbf{a}}=\emptyset$ or $\Delta_{(\mathbf{a},t)}(J^n)=\Delta_{\mathbf{a}}(I^{n-t})$ by (2) and (3) of Proposition~\ref{c2}. From this  it  follows that $$\dim_K\widetilde{H}_{i-|G_{(\mathbf{a},t)}|-1}(\Delta_{(\mathbf{a},t)}(J^n);K)\geq \dim_K\widetilde{H}_{i-|G_{\mathbf{a}}|-1}(\Delta_{\mathbf{a}}(I^{n-t});K)\neq 0$$ in view of Lemma~\ref{sh1}.(4), and so $H_{\mathfrak{n}}^i(R/J^n)_{(\mathbf{a},t)}\neq 0$. Thus, we also have $a_i(R/J^n)\geq a_i(S/I^{n-t})+t$.

Conversely, we may harmlessly assume that $a_i(R/J^n)\neq -\infty$.  Let $(\mathbf{a},t)\in \mathbb{Z}^{r+1}$ such that $H_{\mathfrak{n}}^i(R/J^n)_{(\mathbf{a},t)}\neq 0$ and $a_i(R/J^n)=|\mathbf{a}|+t.$   It is clear that $\Delta_{(\mathbf{a},t)}(J^n)\neq \emptyset$, and it is also clear that $t\leq n-1$ by Proposition~\ref{c2}.(4).

If $i\geq 2$,  then it must be $t\geq 0$, for otherwise, we have  $\Delta_{(\mathbf{a},t)}(J^n)=\{\emptyset\}$ and $G_{\mathbf{a}}=\emptyset$ by Proposition~\ref{c2}.(1), which implies $i=1$ by Lemma~\ref{sh1}.(1), a contradiction. We also have either  $G_{\mathbf{a}}\neq \emptyset$ or $\Delta_{(\mathbf{a},t)}(J^n)\neq \{\emptyset\}$, since $H_{\mathfrak{n}}^i(R/J^n)_{(\mathbf{a},t)}\neq 0$.
  From these it follows that either $\Delta_{(\mathbf{a},t)}(J^n)=  \Delta_{\mathbf{a}}(I^{n-t})$ or $\Delta_{(\mathbf{a},t)}(J^n)=\Delta_{\mathbf{a}}(I^{n-t})\cup  \{r+1\}$ and $G_{\mathbf{a}}=\emptyset$ by Proposition~\ref{c2}.(2,3). Therefore, $\widetilde{H}_{i-|G_{\mathbf{a}}|-1}(\Delta_{\mathbf{a}}(I^{n-t});K)\neq 0$ in view of Lemma~\ref{sh1}.(4). Since $G_{\mathbf{a}}=G_{(\mathbf{a},t)}\in \Delta(J)\cap 2^{[r]}=\Delta(I)$, we have $H_{\mathfrak{m}}^i(S/I^{n-t})_{\mathbf{a}}\neq 0$ by Lemma~\ref{Takayama} , and thus, $a_i(R/J^n)\leq a_i(S/I^{n-t})+t$,  completing the proof of (1).

Now consider the case when $i=1$. If $t<0$ then it must be $\Delta_{(\mathbf{a},t)}(J^n)=\{\emptyset\}$ and $G_{\mathbf{a}}=\emptyset$ with $|\mathbf{a}|\leq n-1$ by Proposition~\ref{c2}.(1). Thus, $a_1(R/J^n)=|\mathbf{a}|-|t|\leq n-2$.

Suppose that  $t\geq 0$.  If $\Delta_{(\mathbf{a},t)}(J^n)=\{\emptyset\}$, then it must be $|G_{\mathbf{a}}|=1$ and so $\Delta_{\mathbf{a}}(I^{n-t})=\{\emptyset\}$ by Proposition~\ref{c2}.(2). Since $G_{(\mathbf{a},t)}\in \Delta(J)$, we have $G_{\mathbf{a}}\in \Delta(I)$. From these it follows that $H_{\mathfrak{m}}^1(S/I^{n-t})_{\mathbf{a}}\neq 0$ by Lemma~\ref{Takayama} and so $a_1(R/J^n)\leq a_1(S/I^{n-t})+t$.

If $\Delta_{(\mathbf{a},t)}(J^n)\neq \{\emptyset\}$ then either $\Delta_{(\mathbf{a},t)}(J^n)=  \Delta_{\mathbf{a}}(I^{n-t})$ or $\Delta_{(\mathbf{a},t)}(J^n)= \Delta_{\mathbf{a}}(I^{n-t})\cup \{r+1\}$ and $G_{\mathbf{a}}=\emptyset$ by (2) and (3) of Proposition~\ref{c2}. In the first case, we have $H_{\mathfrak{m}}^1(S/I^{n-t})_{\mathbf{a}}\neq 0$ and $a_1(R/J^n)\leq a_1(S/I^{n-t})+t$; In the second case, it follows that  $|\mathbf{a}|\leq n-1$ by Proposition~\ref{c2}.(3) and so $a_1(R/J^n)=|\mathbf{a}|+t\leq 2n-2$.

Finally, we show that $a_1(R/J^n)\geq 2n-2$. We may harmlessly assume that $I[1]\neq S[1]$. Let $\mathbf{a}$ denote the vector $(n-1,0,\ldots,0)$ of $\mathbb{Z}^{r}$. Since $$x_1^{n-1}y^{n-1}\notin (J[r+1])^nS=(x_1,\ldots,x_r)^nS$$ and  $$x_1^{n-1}y^{n-1}\notin (J[1])^n=(I[1],y)^nS,$$ both $\{r+1\}$ and $\{1\}$
are faces of $\Delta_{(\mathbf{a},n-1)}(J^n)$. On the other hand, for any $j\in [r]$, since $J[j,r+1]R=R$, one has  $\{j,r+1\}\notin \Delta_{\mathbf{a}}(J^n)$. From this it follows that  $\Delta_{(\mathbf{a},n-1)}(J^n)$ is disconnected and so $a_1(R/J^n)\geq |\mathbf{a}|=2n-2$, as required. \end{proof}
The  condition that $I[j]\neq S[j]$ for some $j\in [r]$  is equivalent to that $\sqrt I\neq (x_1,\ldots, x_r)$.  It  is a weak requirement but necessary in Theorem~\ref{2.12}.(2) in view of Example~\ref{Example}.

\section{The $a_1$ and $a_2$ of $S/I_G^n$}

 By a simple graph   we mean an undirected graph  having no loops and no parallel edges. In this section we always assume that  $G$ is a simple graph without isolated vertex on  vertex set $[r]$. Let $E(G)$ denote the edge set of $G$, and let  $I_G$  denote the Stanley-Reisner ideal of $G$ when $G$ is considered as an one-dimensional simplicial complex. This means $$I_G=(\mathbf{x}_F:  F\subseteq [r], 2\leq |F|\leq 3, F\notin E(G) )=\bigcap_{e\in E(G)} P_e \subseteq K[x_1,\ldots,x_r].$$
  Here $\mathbf{x}_F$ denotes the monomial $\prod_{i\in F}x_i$ and $P_e$ the monomial prime  $(x_i:i\in [r]\setminus e)$.  We will compute the values of $a_1(S/I_G^n)$ and $a_2(S/I_G^n)$ for $n\geq 1$. For the convenience, we set  $$a_i^j(S/I_G^n):=\max\{|\mathbf{a}|: \mathbf{a}\in \mathbb{Z}^r, H_{\mathfrak{m}}^i(S/I_G^n)_{\mathbf{a}}\neq 0, |G_{\mathbf{a}}|=j\}$$ for $i,j\geq 0$. Here $|\mathbf{a}|=a_1+\cdots+a_r$ and $|G_{\mathbf{a}}|$ is the cardinality of  $G_{\mathbf{a}}$.
Thus, $$a_i(S/I_G^n)=\max\{a_i^j(S/I_G^n):j\geq 0\}$$ for all $i\geq 0$.

We recall some basic  notions in graph theory.  Let $p\in [r]$. We use $N_p$ to denote the neighborhood of $p$, that is, $N_p:=\{i\in [r]: \{i,p\}\mbox{  is an edge of  } G \}$.  The {\it degree} of $p$, denoted by $\deg p$, is the cardinality of $N_p$. The maximal  degree  of vertices of $G$ is denoted by $\deg(G)$. It is clear that if $\deg(G)=1$ then $G$ is the disjoint union of some edges. Let $q\in [r]$.  The {\it path} between $p$ and  $q$ is a sequence of distinct vertices $p=p_0,p_1,\ldots,p_{\ell}=q$ such that $\{p_i,p_{i+1}\}$ is an edge of $G$ for $i=0,\ldots,\ell-1$. We write this path as $p=p_0-p_1-\cdots-p_{\ell}=q$ and call $\ell$ to be its {\it length}.  The {\it distance} between  $p$ and $q$, denoted by $d_G(p,q)$ or just $d(p,q)$, is the minimal length of paths from $p$ to $q$, with the convention  that  $d_G(p,q)=\infty$ if there is no paths connecting $p$ and $q$. The maximum of $d_G(p,q)$ with $p,q\in [r]$ is called the {\it diameter} of $G$, denoted by $\mbox{diam}(G)$. Thus,  $\mbox{diam}(G)=\infty$ if and only if $G$ is disconnected.

 Let $\ell \geq 3$. By a {\it cycle} of length $\ell$, we mean a sequence of vertices $p_1,p_2,\ldots,p_{\ell}$ such that $\{p_{i},p_{i+1}\}$ is an edge of $G$ for $i=1,\ldots,\ell-1$ with $p_1,\ldots,p_{\ell-1}$ pairwise distinct and $p_1=p_{\ell}$.  The {\it girth} of a simple graph $G$, denoted by $\mbox{girth}(G)$,  is the smallest length of cycles of $G$, with the convention that $\mbox{girth}(G)=\infty$ if $G$ contains no cycles. It is clear that $3\leq \mbox{girth}(G)\leq \infty$ for any simple graph $G$, and $\mbox{girth}(G)= \infty$ if and only if $G$ is a forest. For unexplained terminology  in graph theory we
refer to \cite{W}.

\subsection{The computation of $a_1(S/I_G^n)$} In this subsection we will compute $a_1(S/I_G^n)$.

 It is clear that if $I$ is generated by monomials $u_1, u_2, \ldots, u_k$ then the monomial localizaton $I[F]=\pi_F(I)$ is generated by $u_1[F],\ldots,u_k[F]$,   where $u_i[F]$ is the image of $u_i$ under the map $\pi_F$ for $i=1,\ldots,k$. If $F=\{p_1,\ldots,p_s\}$ we write $I[p_1,\ldots,p_s]$ instead of $I[\{p_1,\ldots,p_s\}]$. We begin with some   descriptions of monomial localizations of $I_G$ and the simplicial complexes  $\Delta_{\mathbf{a}}(I_G^n)$ for $\mathbf{a}\in \mathbb{N}^r$.

\begin{Lemma} \label{3.1} Let $G$ be a simple graph on $[r]$ and let $\{p,q\}$ be an edge of $G$.  Then  $$I_G[p,q]=(x_i:i\in [r]\setminus \{p,q\}).$$
\end{Lemma}
\begin{proof} It is immediate from the definition of $I_G[p,q]$.
\end{proof}
Denote by $E(G)$ the edge set of $G$. It is clear that if $\{p,q\}\notin E(G)$ then $1\in I[p,q]$, i.e., $I[p,q]=S[p,q]$, and so $\{p,q\}\notin \Delta_{\mathbf{a}}(I_G^n)$ for any $\mathbf{a}\in \mathbb{Z}^r$ and $n\geq 1$.
We now describe the faces of dimension 1 of $\Delta_{\mathbf{a}}(I_G^n)$ for $\mathbf{a}\in \mathbb{N}^r$ and for $n\geq 1$.

\begin{Proposition} \label{edge} Let $p,q$ be distinct vertices of $G$ and $\mathbf{a}=(a_1,\ldots,a_r)$ a vector of $\mathbb{N}^r$. Fix $n\geq 1$. Then the following statements are equivalent:

{\em (1) } $\{p,q\}\in \Delta_{\mathbf{a}}(I_G^n)$;

{\em (2)}  $\{p,q\}\in E(G)$ and $\sum_{i\in [r]\setminus \{p,q\}}a_i\leq n-1$.
\end{Proposition}
\begin{proof} It is immediate from Lemma~\ref{3.1} together with  Lemma~\ref{local}.
\end{proof}

For a simplicial complex  $\Delta$ and an integer $i\geq 0$,  recall from \cite[Page 144]{HH} that the pure $i$th skeleton of $\Delta$ is defined to be the  pure simplicial complex $\Delta(i)$ whose facets are the faces $F$ of $\Delta$ with $|F|=i+1$. By \cite[Lemmas 1.3 and 2.1]{MT} we see that $\Delta_{\mathbf{a}}(I_G^n)(1)$ coincides with $\Delta_{\mathbf{a}}(I_G^{(n)})$.

Unlike the case of symbolic powers, $\Delta_{\mathbf{a}}(I_G^n)$ may contain  a facet of dimension zero. Let $p\in [r]$.  We say that $p$ is an {\it isolated vertex} of $\Delta_{\mathbf{a}}(I_G^n)$ if $\{p\}$ is a facet of $\Delta_{\mathbf{a}}(I_G^n)$. We want to know when a vertex $p$ of $G$ is an isolated one of $\Delta_{\mathbf{a}}(I_G^n)$.   For this, we denote  $M_p:=[r]\setminus (N_p\cup\{p\})$.

\begin{Lemma} \label{3.3} Let $G$ be a simple graph on $[r]$ and $p$  a vertex of degree $\geq 2$.  Then  $$I_G[p]=(x_ix_j:i\neq j \mbox{ and } i,j\in N_p)+(x_i: i\in M_p).$$
\end{Lemma}

\begin{proof} It is clear from the definition of $I_G[p]$.
\end{proof}

 \begin{Proposition} \label{vertex} Let $p$ be a vertex of degree $\geq 2$, and $\mathbf{a}=(a_1,\ldots,a_r)$ a vector in $\mathbb{N}^r$. Fix $n\geq 1$.  Then the following statements are equivalent:

{\em (1)} The vertex $p$ is an isolated vertex of $\Delta_{\mathbf{a}}(I_G^n)$;

 {\em (2)} There exists $t\in \{0,\ldots,n-1\}$ such that

 \begin{itemize}\item $\sum_{i\in M_p}a_i=t$; \hfill\ding{202}
 \item $\sum_{i\in N_p\setminus\{j\}}a_i \geq n-t$ for all $j\in N_p$; \hfill \ding{203}

 \item $\sum_{i\in N_p} a_i\leq 2(n-t)-1$.\hfill \ding{204}
 \end{itemize}
\end{Proposition}

\begin{proof} Set $J:=(x_ix_j:i\neq j \mbox{ and } i,j\in N_p)$ and $K:=(x_i: i\in M_p)$.  Thus $I_G[p]=J+K$ by Lemma~\ref{3.3}. Since $J$ and $K$ have disjoint supports, we have $\mathbf{x}^{\mathbf{a}}\in (I_G[p])^n$ if and only if there exists $t\in \{0,\ldots, n\}$ such that $\mathbf{x}^{\mathbf{a}}\in J^{n-t}$ and $\mathbf{x}^{\mathbf{a}}\in K^t$. Denote by $t$ the number $\sum_{i\in M_p}a_i$.

 (1)$\Rightarrow$ (2) Assume that $p$ is an isolated vertex of $\Delta_{\mathbf{a}}(I_G^n)$.  If $t\geq n$, then $\mathbf{x}^{\mathbf{a}}\in K^n\subseteq (I_G[p])^n$, and so $p\notin \Delta_{\mathbf{a}}(I_G^n)$,  a contradiction. Thus, $t\in \{0,\ldots,n-1\}$. For each $j\in N_p$, since $\{p,j\}\notin \Delta_{\mathbf{a}}(I_G^n)$, it follows that $\sum_{i\in [r]\setminus \{p,j\}}a_i\geq n$ by Proposition~\ref{edge}. This is equivalent to  requiring $\sum_{i\in N_p\setminus \{j\}}a_i\geq n-t$. It remains to be shown the inequality \ding{204} holds.

 First, we show that $a_i\leq n-t-1$ for all $i\in N_p$. In fact, if $a_j\geq n-t$ for some $j\in N_p$, then, since $\sum_{i\in N_p\setminus \{j\}}a_i\geq n-t$,   we have $\mathbf{x^a}\in J^{n-t}$, and so $\mathbf{x^a}\in I^n_G[p]$. This implies $p\notin \Delta_{\mathbf{a}}(I_G^n)$,  contradicting to our assumption.
Thus, we have $a_i\leq n-t-1$ for all $i\in N_p$. Due to this fact,  in order to obtain the  inequality \ding{204} it is enough to prove the following statement:

{\it Let $t\in \{0,\ldots,n-1\}$. If\ $\sum_{i\in N_p}a_i\geq 2(n-t)$ and $a_i\leq n-t-1$ for each $i\in N_p$,   then $\mathbf{x^{a}}\in J^{n-t}.$}

Set $k:=n-t$.  We will proceed by  induction on $k$. If $k=1$ there is nothing to prove; if $k=2$ then $\mathbf{x^{a}}$ can be written as $x_ix_jx_kx_{\ell}u$, where  $i,j,k,\ell\in N_p$ are pairwise distinct and   $u$ is a monomial in $S$. It follows that $\mathbf{x^{a}}\in J^2$, as required. Suppose now that $k\geq 2$. Since  $\sum_{i\in N_p}a_i\geq 2k$, we may write $\mathbf{a}=\mathbf{b}+\mathbf{c}$ such that $\mathbf{b}=(b_1,\ldots,b_r)\in \mathbb{N}^r,\mathbf{c}\in \mathbb{N}^r$ and $\sum_{i\in N_p}b_i=2k$.
We may harmlessly assume further $N_p=\{1,2,\ldots,s\}$ and $b_1\geq b_2\geq \cdots\geq b_s$. Since $b_i\leq a_i\leq k-1$ for $i=1,\ldots,s$,  we have both $b_1$ and $b_2$ are positive integers, and $b_i\leq k-2$ for $i=3,\ldots,s$. Let $\mathbf{b}'$  denote the vector $(b_1-1,b_2-1,b_3,\ldots,b_s,b_{s+1},\ldots, b_r)\in \mathbb{N}^r$. Then $\mathbf{x}^{\mathbf{b}'}\in J^{k-1}$ by the induction hypothesis.  From this it follows that $\mathbf{x}^{\mathbf{a}}=\mathbf{x}^{\mathbf{b}'}x_1x_2\mathbf{x}^{\mathbf{c}}\in J^k$. Thus, the  desired statement has been proved, and the proof of (1)$\Rightarrow$ (2) is now  complete.

(2)$\Rightarrow$ (1) From the  inequalities \ding{202} and \ding{203}, it follows that $\sum_{i\in [r]\setminus \{p,j\}}a_i\geq n$ for any $j\in N_p$. Thus, $\{p,j\}$ is not an edge of $\Delta_{\mathbf{a}}(I_G^n)$ for any $j\in N_p$. Since $J^{n-t}$ is generated in degree $2(n-t)$, $\mathbf{x}^{\mathbf{a}}\notin J^{n-t}$ by the inequality \ding{204}. This then implies $\mathbf{x}^{\mathbf{a}}\notin (I_G[p])^n$ and so $p$ is an isolated vertex of $\Delta_{\mathbf{a}}(I_G^n)$.
\end{proof}

\begin{Corollary} \label{3.5} {\em (1)} Let $p\in [r]$. If  $\deg p\leq 2$ then $p$ is never  an isolated vertex of $\Delta_{\mathbf{a}}(I_G^n)$ for any $\mathbf{a}\in \mathbb{N}^r$ and $n\geq 1$.

{\em (2)} $\Delta_{\mathbf{a}}(I_G)$ contains no isolated vertices for any $\mathbf{a}\in \mathbb{N}^r.$
\end{Corollary}
\begin{proof}
(1) If  $\deg p=1$, then $I_G[p]=I_G[p,q]=(x_i: i\in [r]\setminus \{p,q\})$, where $q$ is the unique vertex belonging to $N_p$. If $p$ is an isolated vertex  of $\Delta_{\mathbf{a}}(I_G^n)$ then it implies that $\mathbf{x}^{\mathbf{a}}\notin (I_G[p])^n$ and $\mathbf{x}^{\mathbf{a}}\in (I_G[p,q])^n$ at the same time, which is impossible.

If $\deg p=2$,  then the three conditions in Proposition~\ref{vertex} for $p$ to be an isolated vertex can not be fulfilled at the same time. In fact, if we write $N_p=\{1,2\}$, then  $a_i\geq n-t$ for  $i=1,2$ by \ding{203}, and it follows that $\sum_{i\in N_p}a_i =a_1+a_2\geq 2(n-t)$. This is contradicted with \ding{204}.

(2) Assume  that $p$ is an isolated vertex of $\Delta_{\mathbf{a}}(I_G)$.   Then $\sum_{i\in N_p} a_i\leq 1$, and $\sum_{i\in N_p\setminus\{j\}} a_i\geq 1$ for all $j\in N_p$ by Proposition~\ref{vertex}. But this is impossible since $|N_p|\geq 3$ by (1).
\end{proof}

Hereafter we   evaluate $a_1^i(S/I_G^n)$ for $i\geq 0$.   If  $i\geq 2$ then $a_1^i(S/I_G^n)=-\infty$  by Lemma~\ref{lu} and by the obvious fact that $\widetilde{H}_i(\Delta;K)=0$ for all $i\leq -2$ and for all $\Delta$.  In the following result we compute $a_1^1(S/I_G^n).$

\begin{Proposition} \label{1} {\em  (1)} If either $\deg(G)\leq 2$ or $n=1$, then $a_1^1(S/I_G^n)=-\infty$.

{\em (2)} If $\deg(G)\geq 3$ and $n\ge2$, then $a_1^1(S/I_G^n)=2n-2$.
\end{Proposition}
\begin{proof} Let $\mathbf{a}\in \mathbb{Z}^r$  with $|G_{\mathbf{a}}|=1$. Then  $H_{\mathfrak{m}}^1(S/I_G^n)_{\mathbf{a}}\neq 0$ if and only if $\Delta_{\mathbf{a}}(I_G^n)=\{\emptyset\}$ by Lemma~\ref{lu}. The latter is equivalent to requiring $p$ is an isolated vertex $\Delta_{\mathbf{a}_+}(I_G^n)$ by Lemma~\ref{transfer}, where $p$ is the unique element of $G_{\mathbf{a}}$. Thus, if either $\deg (G)\leq 2$ or $n=1$, then $a_1^1(S/I_G^n)=-\infty$
by Corollary~\ref{3.5}.

If $n\geq 2$ and $\deg p\geq 3$ for some $p\in [r]$, we may harmlessly assume $\{1,2,3\}\subseteq N_p$. Then $p$ is an isolated vertex $\Delta_{\mathbf{a}_+}(I_G^n)$ by Proposition~\ref{vertex}, where $$\mathbf{a}=(n-1,n-1,1,0,\ldots,-1,\ldots,0)$$ with ``$-1$" appearing at the $p$-th position. From this it follows that $a_1^1(S/I_G^n)\geq 2n-2$. Finally, we obtain $a_1^1(S/I_G^n)\leq 2n-2$ by the  inequalities \ding{202} and \ding{204} in the statement of  Proposition~\ref{vertex}.
\end{proof}
The value of $a_1(S/I_G^n)$ when $n=1$ is now clear.
\begin{Corollary} \label{n=1} If $G$ is disconnected, then $a_1(S/I_G)=0$; If $G$ is connected, then $a_1(S/I_G)=-\infty$.                                                                                 \end{Corollary}
\begin{proof} Let $\mathbf{a}\in \mathbb{N}^r$. If $H_{\mathfrak{m}}^1(S/I_G)_{\mathbf{a}}\neq 0$, then $\Delta_{\mathbf{a}}(I_G)$ is disconnected, and so $\Delta_{\mathbf{a}}(I_G)$  contain disjoint edges due to Corollary~\ref{3.5}.(2).   From this it follows that $\mathbf{a}=(0,\ldots,0)$ by Proposition~\ref{edge} and $a_1^0(S/I_G)\in\{0,-\infty\}$.

Denote $(0,\ldots,0)$ by $\mathbf{0}$. Since $\Delta_{\mathbf{0}}(I_G)=G$, $\Delta_{\mathbf{0}}(I_G)$ is disconnected if and only if $G$ is disconnected.  Thus, $a_1^0(S/I_G)=0$ if $G$ is disconnected, and $a_1^0(S/I_G)=-\infty$ if $G$ is connected. Now, the result follows from Proposition~\ref{1}.(1).
\end{proof}

We remark that Corollary~\ref{n=1} is also clear from  \cite[Lemma 2.1 and 2.2]{HT2}.
Next, we  compute $a_1^0(S/I_G^n)$ for $n\geq 2$.

\begin{Definition} {\em Let $G$ be a simple graph on [r]. A  vertex $p$ of $G$ is said to be} compact {\em if $\deg p\geq 3$ and it belongs to a triangle, where a triangle is a cycle of length 3.}
\end{Definition}

One may look at Broom as given in Figure~\ref{G}, where $3$ is a compact vertex.

\begin{Proposition} \label{nogood} Let $G$ be a simple graph on $[r]$ with $r\geq 3$.  Suppose that $G$ contains no compact vertices.  Then, for all $n\geq 2$, we have $a_1^0(S/I_G^n)\leq 2n-1$.  Moreover, $a_1^0(S/I_G^n)= 2n-1$ if and only if there exist   non-adjacent distinct vertices $p$ and $q$ of $G$ such that $|N_p\cap N_q|\geq 3$.
\end{Proposition}
\begin{proof} We first  prove $a_1^0(S/I_G^n)\leq 2n-1$. If $a_1^0(S/I_G^n)=-\infty$, there is nothing to prove. Suppose now  that $a_1^0(S/I_G^n)\neq -\infty$ and let $\mathbf{a}=(a_1,\ldots,a_r)\in \mathbb{N}^r$ such that $H_{\mathfrak{m}}^1(S/I_G^{n})_{\mathbf{a}}\neq 0$ and $a_1^0(S/I_G^n)=|\mathbf{a}|$. Then $\Delta_{\mathbf{a}}(I_G^n)$ is disconnected by Lemma~\ref{lu} as well as Lemma~\ref{sh1}.(3). We consider the following three cases.

\noindent{\it   Case 1:} There exist disjoint edges of $G$, say $\{1,2\}$ and $\{3,4\}$,  which belong to distinct connected components of $\Delta_{\mathbf{a}}(I_G^n)$.  Then, by Proposition~\ref{edge}, we obtain $$a_3+a_4+a_5+\cdots+a_r\leq n-1$$ and $$a_1+a_2+a_5+\cdots+a_r\leq n-1.$$
In particular, it follows that  $|\mathbf{a}|\leq 2(n-1)$.

\vspace{2mm}
\noindent{\it   Case 2:}  There exist  an isolated vertex
 and an edge in $\Delta_{\mathbf{a}}(I_G^n)$.  Let $p$ be such an isolated vertex with $N_p$ and $M_p$  defined as before.  Then there exists $t\in \{0,\ldots,n-1\}$ such that $\sum_{i\in M_p}a_i=t$, $\sum_{i\in N_p}a_i\leq 2(n-t)-1$ and $a_i\leq n-t-1$ for all $i\in N_p$ by Proposition~\ref{vertex}.

  Let $e=\{k,\ell\}$ be an edge of $\Delta_{\mathbf{a}}(I_G^n)$. Then $\sum_{i\in [r]\setminus e}a_i\leq n-1$  by Proposition~\ref{edge}. It is clear that $e\nsubseteq N_p$, for  otherwise, we have $p$ is a compact vertex by Corollary~\ref{3.5}, a contradiction. If $e\subseteq M_p$, then $\sum_{i\in e}a_i \leq t\leq n-1$.  From this it follows that $|\mathbf{a}|\leq 2n-2$.  So we may assume $k\in N_p$ and $\ell\in M_p$. Then $a_k\leq n-t-1$ and $a_{\ell}\leq t$, which also implies $|\mathbf{a}|\leq 2n-2$.

\vspace{2mm}
\noindent{\it   Case 3:} $\Delta_{\mathbf{a}}(I_G^n)$ contains  distinct isolated vertices $p$ and $q$. Then, by Proposition~\ref{vertex}, there exist $s,t\in\{0,\ldots,n-1\}$ such that
\[
\begin{split}
        & \mbox{(1)} \mbox{\qquad} \sum_{i\in M_p}a_i=s, \mbox{\qquad} \sum_{i\in M_q}a_i=t;\\
   & \mbox{(2)}\mbox{\qquad} a_i\leq n-s-1, \forall i\in N_p,  \mbox{\qquad} a_i\leq n-t-1, \forall i\in N_q;\\
  & \mbox{(3)}\mbox{\qquad} \sum_{i\in N_p}a_i\leq 2(n-s)-1, \mbox{\qquad}   \sum_{i\in N_q}a_i\leq 2(n-t)-1.
\end{split}
\]
\noindent{\it   Subcase 3.1:}  Suppose first that $\{p,q\}$ is an edge of $G$. Then $p\in N_q$ and $q\in N_p$. We can write $N_p=A\sqcup \{q\}, N_q=B\sqcup \{p\}$.  Here $\sqcup$ means a disjoint union. Note that $A\cap B=\emptyset$ since $G$ contains no compact vertices. It follows that  $M_p=([r]\setminus \{p,q\})\setminus A$ and $M_q=([r]\setminus \{p,q\})\setminus B$, and so    $M_p=B\sqcup C$ and $M_q=A\sqcup C$ for some $C\subseteq [r]$.  Under these notions,  we have $M_p=(N_q\sqcup C)\setminus \{p\}$ and so
     $$\sum_{i\in N_q}a_i=s-\sum_{i\in C}a_i+a_p\leq s-\sum_{i\in C}a_i+n-t-1.$$
    This implies
    \[
    \begin{split}
     |\mathbf{a}|&= \sum_{i\in N_q}a_i+\sum_{i\in M_q}a_i+a_q\leq s-\sum_{i\in C}a_i+n-t-1+t+n-s-1\\
     &=2n-2-\sum_{i\in C}a_i\leq 2n-2.
    \end{split}
     \]
\noindent{\it   Subcase 3.2:}   Suppose next that $\{p,q\}$ is not an edge of $G$. Since $p\in M_q$, it follows that $a_p\leq t$ and so $$|\mathbf{a}|=\sum_{i\in N_p}a_i+\sum_{i\in M_p}a_i+a_p\leq 2(n-s)+s+t-1.  \mbox{\qquad \ding{172}}$$
     Similarly, $$|\mathbf{a}|=\sum_{i\in N_q}a_i+\sum_{i\in M_q}a_i+a_q\leq 2(n-t)+s+t-1. \mbox{\qquad \ding{173}}$$
     Combining \ding{172} and \ding{173}, we conclude that $|\mathbf{a}|\leq 2n-1$ and thus the first statement has been proved.

     Now assume that  $a_1^0(S/I_G^n)= 2n-1$, and let $\mathbf{a}=(a_1,\ldots,a_r)\in \mathbb{N}^r$ such that $H_{\mathfrak{m}}^1(S/I_G^n)_{\mathbf{a}}\neq 0$ and $|\mathbf{a}|=2n-1$.  In view of the proof of the first statement, we have $\Delta_{\mathbf{a}}(I_G^n)$ contains distinct isolated vertices $p$ and $q$,  which are non-adjacent in $G$, and so there exist  $s,t\in\{0,\ldots,n-1\}$ satisfying the inequalities (1), (2), (3) stated before.  Hence, since $p\in M_q$ and $q\in M_p$, we have  $$2n-1=|\mathbf{a}|= \sum_{i\in N_p}a_i+\sum_{i\in M_p}a_i+a_p\leq 2(n-t)-1+s+t$$ and $$2n-1=|\mathbf{a}|=\sum_{i\in N_q}a_i+\sum_{i\in M_q}a_i+a_q\leq 2(n-t)-1+s+t.$$  From these it follows that  $s=t$ and so the inequalities above are actually equalities.
      Therefore,
     $a_p=a_q=t$ and $\sum_{i\in N_p}a_i=\sum_{i\in N_q}a_i=2(n-t)-1$.

\vspace{2mm}
     We now show that $|N_p\cap N_q|\geq 3$.  Set $A:=N_p\cap N_q$. Assume on the contrary that $|A|\leq 2$. Then $\sum_{i\in A}a_i\leq 2(n-t)-2$. It follows that there exists $k\geq 0$   such that $\sum_{i\in A}a_i= 2(n-t)-2-k$ and  $\sum_{i\in N_p\setminus A}a_i=\sum_{i\in N_q\setminus A}a_i=k+1. $  Since $(N_q\setminus A)\sqcup \{q\}\subseteq M_p$, we have $(N_q\setminus A)\sqcup N_p\sqcup \{p,q\} \subseteq [r]$ and this implies

     $$|\mathbf{a}|\geq \sum_{i\in N_q\setminus A}a_i+\sum_{i\in N_p\setminus A}a_i+\sum_{i\in A}a_i+a_p+a_q= 2n+k,$$  a contradiction. Thus, we have proved $|N_p\cap N_q|\geq 3$.

     Conversely, suppose that there exist   non-adjacent  vertices $p$ and $q$   with  $|N_p\cap N_q|\geq 3$. We may assume  $\{1,2,3\}\subseteq N_p\cap N_q$ and let $\mathbf{a}$ denote the vector $$(n-1,n-1,1,0,\cdots,0)\in \mathbb{N}^r.$$ It follows that  $p$ and $q$ are isolated vertices of $\Delta_{\mathbf{a}}(I_G^n)$ by Proposition~\ref{vertex}. This then  implies $a_1^0(S/I_G^n)\geq |\mathbf{a}|=2n-1$, as requied.
\end{proof}

\begin{Proposition} \label{good} If $G$ contains a compact vertex, then $a_1^0(S/I_G^n)=3n-3$ for $n\geq 2$.
\end{Proposition}
\begin{proof}  We may assume that $G$  contains the Broom as given in Figure~\ref{G} as its subgraph. If we set $\mathbf{a}:=(n-1,n-1,n-2,1,0,\cdots,0)$, then $\{1,2\}$ is an edge of $\Delta_{\mathbf{a}}(I_G^n)$ and $3$ is an isolated vertex of $\Delta_{\mathbf{a}}(I_G^n)$ by Propositions~\ref{edge} and \ref{vertex}. In particular, we have $\Delta_{\mathbf{a}}(I_G^n)$ is disconnected and $H_{\mathfrak{m}}^1(S/I_G^n)_{\mathbf{a}}\neq 0$ by Lemma~\ref{lu}. This implies $a_1^0(S/I_G^n)\geq |\mathbf{a}|=3n-3$.

 It remains to be shown that $a_1^0(S/I_G^n)\leq 3n-3$. This is equivalent to showing $|\mathbf{a}|\leq 3n-3$ if $\Delta_{\mathbf{a}}(I_G^n)$ is disconnected (i.e., $H_{\mathfrak{m}}^0(S/I_G^n)_{\mathbf{a}}\neq 0$), where $\mathbf{a}\in \mathbb{N}^r$.  Examining the proof of Proposition~\ref{nogood},  we only need to consider the cases when $\Delta_{\mathbf{a}}(I_G^n)$ contains an isolated vertex $p$ and an edge $e$ with $e\subseteq N_p$ and  when $\Delta_{\mathbf{a}}(I_G^n)$  contains two distinct isolated vertices $p$ and $q$  such that $p$ is adjacent to $q$ in $G$.

 First we consider the case when $\Delta_{\mathbf{a}}(I_G^n)$ contains an isolated vertex $p$ and an edge $e$ with $e\subseteq N_p$. Since $e\in \Delta_{\mathbf{a}}(I_G^n)$ and $e\subseteq N_p$, one has $|\mathbf{a}|-\sum_{i\in e}a_i \leq n-1$ and $a_i\leq n-1$ for $i\in e$. This immediately implies $|\mathbf{a}|\leq 3n-3$.

 For the case when $\Delta_{\mathbf{a}}(I_G^n)$  contains two distinct isolated vertices $p$ and $q$ such that $p$ is adjacent to $q$ in $G$, we note that $N_p,N_q,M_p,M_q$ and  $s,t$  satisfy the inequalities  as given in the proof of Case 3. Then $$|\mathbf{a}|=\frac{1}{2}(\sum_{i\in N_p}a_i+\sum_{i\in M_p}a_i+\sum_{i\in N_p}a_i+\sum_{i\in M_p}a_i+a_p+a_q)\leq 3n-3.$$ This completes the proof.
\end{proof}

\begin{Proposition} \label{diam3} Let $G$ be a simple graph with $\mathrm{diam}(G)\geq 3$. Then $a_1^0(S/I_G^n)\geq 2n-2$ for all $n\geq 2$.
\end{Proposition}

\begin{proof} Since $\mathrm{diam}(G)\geq 3$, there exist a pair of vertices, say $1,2$, such that $d_G(1,2)\geq 3$.  Set $\mathbf{a}:=(n-1,n-1,0,\ldots,0)$.  We claim that $\Delta_{\mathbf{a}}(I_G^n)$ is disconnected. In fact, it is not difficult to see   that $$\mathbf{x^a}=x_1^{n-1}x_2^{n-1}\notin (I^n_G[i])S$$ for $i=1,2$ and so both $\{1\}$ and $\{2\}$ are faces of $\Delta_{\mathbf{a}}(I_G^n)$. If $\Delta_{\mathbf{a}}(I_G^n)$ is connected, then there exits a path $1=p_0-p_1-p_2-\cdots-p_{\ell}=2$ of $\Delta_{\mathbf{a}}(I_G^n)$, connecting $1$ and $2$. Note that this is also a  path of $G$ by  Proposition~\ref{edge}. Since $d_G(1,2)\geq 3$, $\{p_1,p_2\}\cap\{1,2\}=\emptyset$. This implies that $\{p_1,p_2\}\notin \Delta_{\mathbf{a}}(I_G^n)$ by Proposition~\ref{edge} again, a contradiction. From this  our claim follows. Hence $H_{\mathfrak{m}}^1(S/I_G^n)_{\mathbf{a}}\neq 0$ and $a_1^0(S/I_G^{n})\geq |\mathbf{a}|=2n-2$.
\end{proof}

We are now in the position to present the main result of this subsection.  For the simplicity of the statement of this result, we use $\mathfrak{C}_i, i=1,2,\ldots, 5$ to indicate the following five  conditions on  a simple graph $G$ respectively.

 $\mathfrak{C}_1$: $G$ contains a compact vertex;

$\mathfrak{C}_2$: $G$ contains no compact vertices, but it contains two non-adjacent vertices such that the intersection of their neighborhoods contains at least three vertices;

$\mathfrak{C}_3$: $G$ satisfies neither $\mathfrak{C}_1$  nor $\mathfrak{C}_2$, but $\deg(G)\geq 3$;

$\mathfrak{C}_4$: $\mathrm{diam}(G)\geq 3$ and $\deg (G)\leq 2$;

$\mathfrak{C}_5$: $\mathrm{diam}(G)\leq 2$ and $\deg (G)\leq 2$.

Note that if  $\deg (G)\leq 2$, then $G$ is the disjoint union of some paths and some cycles. Denote by $C_{\ell}$ and $P_{\ell}$ a cycle  and a path of length $\ell$, respectively.  Thus, if $G\in \mathfrak{C}_5$ then $G\in \{P_1,P_2, C_4,C_5\}$.

\begin{Proposition}\label{main3.1} Let $G$ be a simple graph on $[r]$ with $r\geq 3$.  Then, for all $n\geq 2$, we have $$a_1(S/I_G^n)=\left\{
                                                                                   \begin{array}{ll}
                                                                                     3n-3, & \hbox{$G\in \mathfrak{C}_1$;} \\
                                                                                     2n-1, & \hbox{$G\in \mathfrak{C}_2$;} \\
                                                                                     2n-2, & \hbox{$G\in \mathfrak{C}_3$;} \\
  2n-2, & \hbox{$G\in \mathfrak{C}_4$;} \\
2n-2, &\hbox{$G=C_5$  and $n\geq 3$;}\\

-\infty, &\hbox{$G=C_5$  and $n= 2$;}\\
 -\infty, & \hbox{$G\in \{P_2,C_4\}.$}
                                                                                   \end{array}
                                                                                 \right.
$$

\end{Proposition}

\begin{proof} If $G\in \mathfrak{C}_1$ then $a_1(S/I_G^n)=3n-3$ by Proposition~\ref{good}. If $G\in \mathfrak{C}_2$ then $a_1(S/I_G^n)=2n-1$ by Proposition~\ref{nogood}.  If either  $G\in \mathfrak{C}_3$ or $G\in \mathfrak{C}_4$ then $a(S/I_G^n)=2n-2$ by Lemmas~\ref{1} and \ref{diam3} together with Proposition~\ref{nogood}. It is left to consider the case when $G\in \mathfrak{C}_5$, namely, $G\in \{P_2,C_4,C_5\}$.  The case when $G\in \{P_2,P_4\}$  is proved in a similar way as in the case when $G=C_5$  and we omit its proof.

Suppose  now that $G$ is 5-cycle, that is,  $G$ is the  Pentagon in Figure~\ref{G}. Then  $H_{\mathfrak{m}}^1(S/I_G^n)_{\mathbf{a}}\neq 0$ if and only if $\mathbf{a}\in \mathbb{N}^r$ and there exist two disjoint edges of $G$,  which belongs to distinct connected components of $\Delta_{\mathbf{a}}(I_G^n)$. This implies $|\mathbf{a}|\leq 2n-2$, by the same argument using in the proof of Case 1 of Proposition~\ref{nogood} and so $a_1(S/I_G^n)\leq 2n-2$. If $n\geq 3$, we set $\mathbf{a}:=(1,n-2,0,n-2,1)$. Then $\Delta_{\mathbf{a}}(I_G^n)=\langle\{1,2\},\{4,5\}\rangle$   is disconnected and so  $a_1(S/I_G^n)= 2n-2$. It remains to be shown that if $n=2$ then $\Delta_{\mathbf{a}}(I_G^n)$ is connected for any $\mathbf{a}\in \mathbb{N}^r$. If not, we may harmlessly assume that $\{1,2\}$ and $\{3,4\}$ belong to different connected components of $\Delta_{\mathbf{a}}(I_G^n)$. Note that $\{2,3\}\notin \Delta_{\mathbf{a}}(I_G^n)$, we have $a_1=a_4=1$ and $a_2=a_3=a_5=0$ by Lemma~\ref{edge}. From this it follows that both $\{1,5\}$ and $\{4,5\}$ belong to $\Delta_{\mathbf{a}}(I_G^n)$ and thus $\Delta_{\mathbf{a}}(I_G^n)$ is connected, a contradiction. This completes the proof.
\end{proof}

  Since $\mbox{Krull-}\dim S/I_G^n=2$, we see that $S/I_G^n$ is Cohen-Macaulay if and only if $H^0_{\mathfrak{m}}(S/I_G^n)=H^1_{\mathfrak{m}}(S/I_G^n)=0$, namely, $a_0(S/I_G^n)=a_1(S/I_G^n)=-\infty$. The following result recovers \cite[Corollaries 3.4 and 3.5]{MT}.

\begin{Corollary} Let $G$ be a simple graph on $[r]$ with $r\geq 3$ and $n$ an integer $\geq 1$.
Then the following statements are equivalent:

{\em (1)} $S/I_G^n$ is Cohen-Macaulay;

{\em (2)} Either  $G\in \{P_2,C_4\}$ or $G=C_5$ and $n=2$.
\end{Corollary}

\begin{proof} If (1) holds then $a_1(S/I_G^n)=-\infty$,  and thus  we obtain (2) by Proposition~\ref{main3.1}.
Conversely,  we  choose any $G$ and $n$ satisfying (2). Then it is not difficult to check that  $I_G^n=I_G^{(n)}$ and so $S/I_G^n$ has a positive  depth. From this it follows that $H^0_{\mathfrak{m}}(S/I_G^n)=0$.  Since $H^1_{\mathfrak{m}}(S/I_G^n)=0$ by Proposition~\ref{main3.1}, the result follows.
\end{proof}

To illustrate the difference between $a_1(S/I_G^n)$ and $a_1(S/I_G^{(n)})$, we recall a well-known concept in the combinatorial theory.

\begin{Definition} \label{matroid} \em A simple complex $\Delta$ is called a {\it matroid} provided that whenever $F$ and $G$ are faces of $\Delta$ with $|F|<|G|$, there exists $x\in G\setminus F$ such that $F\cup \{x\}$ is also a face of $\Delta$.
 \end{Definition}

 According to \cite[Corollary 2.6]{NT}, a graph $G$, considered as a simple complex,  is  a  matroid if and only if  every pair of disjoint edges of $G$ is contained in a 4-cycle.
   It was proved in \cite{MT} that if $G$ is a matroid then $H_{\mathfrak{m}}^1(S/I_G^{(n)})=0$, i.e.,  $a_1(S/I_G^{(n)})=-\infty$. In contrast, the values of  $a_1(S/I_G^n)$ behave very differently when $G$ is a matroid.

\begin{Example} {\em  Let $G$ be a matroid on $[r]$ with $r\geq 3$. Then $a_1(S/I_G^n)$ could be any number of $\{3n-3, 2n-1, 2n-2, -\infty\}$, depending on the structure of $G$. In fact, according to Proposition~\ref{main3.1}, we see that if $G$ is a complete graph on $[4]$ then $a_1(S/I_G^n)=3n-3$; if $G$ is the graph on $[5]$ with $E(G)=\{\{4,i\}, \{5,i\}: i=1,2,3\}$ (such a graph is called a diamond in the lattice theory) then $a_1(S/I_G^n)=2n-1$; if $G$ is the graph on [4] with $E(G)=\{\{4,i\}:i=1,2,3\}$ then  $a_1(S/I_G^n)=2n-2$. Finally if $G$ is a 4-cycle, then  $a_1(S/I_G^n)=-\infty$. }

\end{Example}

We end this subsection by  a characterization when a graph is a matroid, which may be of some independent interest. It is clear that  every matroid  has its diameter  $\leq 2$. We now consider obstructions for a graph of $\mbox{diam}(G)\leq 2$ to be a matroid. Let $V\subseteq [r]$. Recall that  the induced subgraph of $G$ on $V$ is the subgraph $G[V]$ on vertex set $V$ such that  for any $i,j\in V$, $i$ is adjacent to $j$ in $G[V]$ if and only if $i$ is adjacent to $j$ in $G$.

\begin{Proposition}\label{obstruction} Let $G$ be a simple graph with $\mathrm{diam}(G)\leq 2$. Then $G$ is not a matroid  if and only if it has an induced subgraph which is isomorphic to either a Broom or a Pentagon, see Figure~\ref{G}.
\end{Proposition}

\begin{figure}[ht!]

\begin{tikzpicture}[line cap=round,line join=round,>=triangle 45,x=1.5cm,y=1.5cm]

\draw (7,1)-- (5,1)--(6,2)--(7,1);
\draw (6,2)--(6,3);

\draw (6,3) node[anchor=south east]{4};

\draw (6,2) node[anchor=south east]{3};

\draw (5,1) node[anchor=south east]{2};
\draw (7.2,1) node[anchor=south east]{1};

\draw (6.3,0.5) node[anchor=north east]{\em{Broom}};

\fill [color=black] (6,3) circle (1.5pt);
\fill [color=black] (7,1) circle (1.5pt);\fill [color=black] (6,2) circle (1.5pt);
\fill [color=black] (5,1) circle (1.5pt);

\draw (11.5,1)--(10.5,1)--(10,2)--(11,3)--(12,2)--(11.5,1);

\draw (11.7,0.5) node[anchor=north east]{\em{Pentagon}};

\fill [color=black] (11.5,1) circle (1.5pt);
\fill [color=black] (10.5,1) circle (1.5pt);
\fill [color=black] (10,2) circle (1.5pt);\fill [color=black] (11,3) circle (1.5pt);\fill [color=black] (12,2) circle (1.5pt);

\draw (11.1,3) node[anchor=south east]{4};

\draw (11.9,1) node[anchor=south east]{1};
\draw (10.4,1) node[anchor=south east]{2};
\draw (10,2) node[anchor=south east]{3};
\draw (12.2,2) node[anchor=south east]{5};
\end{tikzpicture}
\caption{}\label{G}
\end{figure}

\begin{proof} If $G$ has  an induced subgraph isomorphic to either Broom or Pentagon, as given in Figure~\ref{G}, then $\{1,2\}$ and $\{3,4\}$ are disjoint edges that does not belong to any $4$-cycles. Thus, $G$ is not a matroid.  For the proof of the converse, it is enough to show that  there is no a graph $G$ of   $\mathrm{diam}(G)\leq 2$ such that  $G$ is not a matroid  and that $G$ has no induced subgraph isomorphic to one of the graphs as given in Figure~\ref{G}.

Assume on the contrary that such a graph exists. Let $G$ be such a graph. Since $G$ is not matroid, there are disjoint edges, say $\{1,2\}$ and $\{3,4\}$, which does not belong to any 4-cycle. We consider the following cases:

 Suppose that $d_G(i,j)=2$ for any  $i\in \{1,2\}$ and $j\in\{3,4\}$. (This implies none of $\{i,j\}$ with $i\in \{1,2\}$ and $j\in\{3,4\}$ is an edge of $G$).) Let $1-p_1-3$ and $1-p_2-4$ be two paths. If $p_1=p_2$, then, since $\{1,3\}$ and $\{1,4\}$ are not edges of $G$, the subgraph induced on $\{1,3,4, p_1=p_2\}$ is isomorphic to a Broom, a contradiction. Thus, $p_1\neq p_2$. Since $G$ contains no induced subgraphs isomorphic to a pentagon, at least one of the pairs $\{3,p_2\},\{4,p_1\},\{p_1,p_2\}$ is an edge of $G$. If $\{3,p_2\}\in E(G)$, then the subgraph induced on $\{1,p_2,3,4\}$ is isomorphic to a Broom, a contradiction again. From this it follows that $\{3,p_2\}\notin E(G)$.   With the same reason $\{4,p_1\}\notin E(G)$. It must be $\{p_1,p_2\}$  is an edge of $G$. Thus,  the subgraph of $G$ induced on $\{1,p_1,p_2,3\}$ is isomorphic to a Broom. This yields a contradiction again.

 Thus, $d_G(i,j)=1$ for some   $i\in \{1,2\}$ and $j\in\{3,4\}$. Say $d(1,3)=1$, or equivalently, $\{1,3\}$ is an edge of $G$. Then $\{2,4\}$ is not an edge of $G$. Let $2-p-4$ be a path. If $p\in\{1,3\}$, say $p=1$, then since the graph of $G$ induced on $\{1,2,3,4\}$ can not be a Broom, $\{2,3\}$ must be an edge of $G$. Thus $1-2-3-4-1$ is a $4$-cycle of $G$, a contradiction. Consequently, $p\notin \{1,3\}$

 Since the subgraph of $G$ induced on $\{1,2,3,4,p\}$ is not a pentagon,  at least one of $\{1,4\},\{2,3\}, \{1,p\},\{3,p\}$ is an edge of $G$. If $\{1,4\}$ is an edge, since the subgraph of $G$ induced on $\{2,1,3,4\}$ is not a Broom, we have $\{2,3\}\in E(G)$. This implies $1-2-3-4-1$ is a 4-cycle containing $\{1,2\},\{3,4\}$, a contradiction. Thus,  $\{1,4\}$ is not an edge. Similarly  $\{2,3\}\notin E(G)$. If $\{1,p\}\in E(G)$, then subgraph induced on $\{3, 1,2,p\}$ is isomorphic to a Boom, thus $\{1,p\}\notin E(G)$. At last, $\{3,p\}\notin E(G)$. Thus, in all cases, our assumption leads to a contradiction.
\end{proof}

\begin{Corollary} Let $G$ be a simple graph on $[r]$ with $r\geq 3$. Then the following statements are equivalent

{\em (1)} $G$ is a matroid;

{\em (2) $\mbox{diam}(G)\leq 2$}  and $G$ contains neither Broom nor Pentagon as its induced graphs.

\end{Corollary}

\subsection{The computation of $a_2(S/I_G^n)$}
In this subsection we will prove that   $a_2(S/I_G^n)=a_2(S/I_G^{(n)})$ and give the formula of $a_2(S/I_G^n)$.

\begin{Proposition} \label{a21} Let $G$ be a simple graph on $[r]$ with $r\geq 3$. Then $$a_2^j(S/I_G^n)=a_2^j(S/I_G^{(n)})$$ for all $j\geq 0$ and $n\geq 1$.  In particular, $a_2(S/I_G^n))=a_2(S/I_G^{(n)})$ for $n\geq 1$.
\end{Proposition}

\begin{proof} We have known that $\Delta_{\mathbf{a}}(I_G^n)(1)=\Delta_{\mathbf{a}}(I_G^{(n)})$ by Lemma~\ref{edge}. Let $T\in \{I_G^n, I_G^{(n)}\}$.
     If $\mathbf{a}$ is a vector in $\mathbb{Z}^r$ with $|G_{\mathbf{a}}|=2$, then
$$H_{\mathfrak{m}}^2(S/T)_{\mathbf{a}}\neq 0\Longleftrightarrow\Delta_{\mathbf{a}}(T)=\{\emptyset\}\Longleftrightarrow G_{\mathbf{a}}\in \Delta_{\mathbf{a}_{+}}(T).$$
Here, the first equivalence follows from Lemma~\ref{lu} as well as Lemma~\ref{sh1}.(1), and the second one is from Lemma~\ref{transfer}.
On the other hand, we have $G_{\mathbf{a}}\in \Delta_{\mathbf{a}_{+}}(I_G^n)$ if and only if $G_{\mathbf{a}}\in \Delta_{\mathbf{a}_{+}}(I_G^{(n)})$ by Lemma~\ref{edge}. This proves $a_2^2(S/I_G^n)=a_2^2(S/I_G^{(n)})$.

Let  $\mathbf{a}$ be a vector in $\mathbb{Z}^r$ with $|G_{\mathbf{a}}|=1$. Say $G_{\mathbf{a}}=\{p\}$. Then
$H_{\mathfrak{m}}^2(S/T)_{\mathbf{a}}\neq 0$ if and only if $\Delta_{\mathbf{a}}(T)$ is disconnected. This is equivalent to requiring that there exist $i\neq j\in [r]\setminus \{p\}$ such that  $\{p,i\}$ and $\{p,j\}$ belong to $ \Delta_{\mathbf{a}_+}(T)$ by Lemma~\ref{transfer}.
Since  $\{p,k\}\in \Delta_{\mathbf{a}_{+}}(I_G^n)$ if and only if $\{p,k\}\in \Delta_{\mathbf{a}_{+}}(I_G^{(n)})$ for $k=i,j$ by Lemma~\ref{edge}, the equality  $a_2^1(S/I_G^n)=a_2^1(S/I_G^{(n)})$ follows.

Let  $\mathbf{a}\in \mathbb{N}^r$.  Then
$H_{\mathfrak{m}}^2(S/T)_{\mathbf{a}}\neq 0\Longleftrightarrow\Delta_{\mathbf{a}}(T) \mbox{ contains a cycle  }$ by Lemma~\ref{sh2}.
It is clear that  $\Delta_{\mathbf{a}}(I_G^n)$  contains a cycle  if and only if $\Delta_{\mathbf{a}}(I_G^{(n)})$ contains a cycle. From this it follows that $a_2^0(S/I_G^n)=a_2^0(S/I_G^{(n)})$.
Note that $a_2^j(S/I_G^n)=a_2^j(I_G^{(n)})=-\infty$ for all $j\geq 3$, our proof completes.
\end{proof}

We can obtain the following formula on $a_2(S/I_G^n)$ immediately by combining Proposition~\ref{a21} with \cite[Theorem 2.8]{HT2}.

\begin{Proposition} \label{a2} Let $G$ be a simple graph on vertex set $[r]$ with $r\geq 3$. Then for all $n\geq 1$,
{\em  $$a_2(S/I_G^n)=\left\{
 \begin{array}{ll}
   3n-3, & \hbox{$\mbox{girth}(G)=3$;} \\
   2n-2, & \hbox{$\mbox{girth}(G)=4$;} \\
   2n-3, & \hbox{$5\leq \mbox{girth}(G)< \infty$  and $n\geq 2$;} \\
    0, &   \hbox{$5\leq \mbox{girth}(G)< \infty$ and $n=1$;}\\
  2n-3, & \hbox{$\mbox{girth}(G)=\infty$ and $\deg (G)\geq  2$;} \\
   n-3, & \hbox{$\mbox{girth}(G)=\infty$ and $\deg(G)=1$.}
                                                                                                                 \end{array}
                                                                                                               \right.
$$}
\end{Proposition}

\section{Conclusions}

If $I$ is a nonzero two-dimensional squarefree monomial ideal of $S=K[x_1,\ldots,x_r]$  containing no variables, then $r\geq 3$ and $I=I_G$,
where $G$ is a simple graph  on  $[r]$ with $E(G)\neq \emptyset$ that may contain isolated vertices.
We now can state and prove  the main result of this paper.

\begin{Theorem} \label{4.1} Let $r\geq 3$ and let $G$ be a simple graph on vertex set $[r]$ with $E(G)\neq \emptyset$ that may contain some isolated vertices. Then, for all $n\geq 1$, we have
{\em  $$ \qquad \mbox{g-reg} (S/I_G^{n})= \left\{
 \begin{array}{ll}
   3n-1, & \hbox{$\mbox{girth}(G)=3$;} \\
   2n, & \hbox{$\mbox{girth}(G)=4$;} \\
   2n-1, & \hbox{$5\leq \mbox{girth}(G)< \infty$ and $n\geq 2$;} \\
   2, & \hbox{$5\leq \mbox{girth}(G)< \infty$ and $n=1$;} \\
  2n-1, & \hbox{$\mbox{girth}(G)=\infty$ and $\deg (G)\geq  2$;} \\
     2n-1, & \hbox{$\deg(G)=1$.}
                                                                                                                 \end{array}
                                                                                                               \right.$$}
\end{Theorem}

\begin{proof} In the process of our proof, Theorem~\ref{2.12} and Propositions~\ref{main3.1},~\ref{a2} will be used many times  without referring to them each time.

Note that $\mbox{g-reg} (S/I_G^{n})=\max\{a_i(S/I_G^n)+i:i=1,2\}$.
If $\mbox{girth}(G)=3$ then $a_1(S/I_G^n)\leq 3n-3$ for all $n\geq 1$ by Proposition~\ref{main3.1}, Corollary~\ref{n=1} and  Theorem~\ref{2.12} and it follows that $\mbox{g-reg}(S/I_G^n)=3n-1$ by Proposition~\ref{a2} and Theorem~\ref{2.12}.  If $\mbox{girth}(G)=4$ then $G$
contains no compact vertices. From this it follows that $a_1(R/I_G^n)\leq 2n-1$ by Propositions~\ref{main3.1} and so $\mbox{g-reg}(S/I_G^n)=2n$ by Proposition~\ref{a2}.

  Assume that $5\leq \mbox{girth}(G)<\infty$. If $n\geq 2$, then, since $G$ does not satisfy $\mathfrak{C}_2$, we have  $a_1(S/I_G^n)\leq 2n-2$. This then implies  $\mbox{g-reg}(S/I_G^n)=2n-1$; If $n=1$ then $a_1(S/I_G)\leq 0$ by Corollary~\ref{n=1} and so $\mbox{g-reg}(S/I_G^n)=2$.

  If $\mbox{girth}(G)=\infty$ and $\deg(G)\geq 2$, we also have  $a_1(S/I_G^n)\leq 2n-2$ and so $\mbox{g-reg}(S/I_G^n)=2n-1$.

Now consider the case when  $\deg(G)=1$. Since $r\geq 3$, either $G$ contains at least two disjoint edges or $G$ contains at least  an edge and an isolated vertex. In the former case, we have  $\mbox{diam}(G)=\infty\geq 3$. From this it follows that $a_1(S/I_G^n)=2n-2$   and so $\mbox{g-reg}(S/I_G^n)=2n-1$. In the latter case, we have $a_1(S/I_G^n)=2n-2$ by Theorem~\ref{2.12}. This also implies $\mbox{g-reg}(S/I_G^n)=2n-1$.
\end{proof}

One special case of Theorem~\ref{4.1} can be achieved  directly: If $G$ is a triangle, then $I_G=(x_1x_2x_3)$ and so $I_G^n$ has a linear resolution for all $n\geq 1$. Since $S/I_G^n$ has positive depth, we have  $$\mbox{g-reg}(S/I_G^n)=\reg(S/I_G^n)=\reg(I_G^n)-1=3n-1.$$

The following result follows immediately by  comparing Theorem~\ref{4.1}   with \cite[Theorem 2.9]{HT2}.

\begin{Corollary} \label{L} Let $G$ be a simple graph on vertex set $[r]$ with $r\geq 3$ that may contain some isolated vertices. If $G$ contains at least one edge, then, for all $n\geq 1$, we have {\em  $$\mbox{g-reg}(S/I_G^n)=\reg(S/I_G^{(n)}).$$}
\end{Corollary}

In general, a two-dimensional squarefree monomial ideal may contain some variables. More precisely, if $I$ is a two-dimensional squarefree monomial ideal which is not generated by variables, then
$I=(I_G,y_1,\ldots,y_k)$, where $k\geq 0$ and  $G$ is again  a simple graph  on vertex set $[r]$ with  $r\geq 3$ and  $E(G)\neq \emptyset$ that may contain isolated vertices.
  We now consider this case. It is not difficult to see that $$S[y]/(I_G,y)^n=S/I_G^n\oplus (S/I_G^{n-1})y\oplus \cdots \oplus (S/I_G)y^{n-1}.$$  By  \cite[Theorem 3.4]{HNTT}, we have
$$S[y]/(I_G,y)^{(n)}=S/I_G^{(n)}\oplus (S/I_G^{(n-1)})y\oplus \cdots \oplus (S/I_G)y^{n-1}.$$ From these it follows that $$\mbox{g-reg}(S[y]/(I_G,y)^n)=\max\{\mbox{g-reg}(S/I_G^i)+n-i:1\leq i\leq n\}$$ and
$$\mbox{reg}(S[y]/(I_G,y)^{(n)})=\max\{\mbox{reg}(S/I_G^{(i)})+n-i:1\leq i\leq n\}.$$
Combining  these equalities with  Corollary~\ref{L},  we get the last result of this paper.

\begin{Corollary} Let $I$ be a two-dimensional squarefree monomial ideal. Then, for all $n\geq 1$, we have {\em  $$\mbox{g-reg}(S/I^n)=\reg(S/I^{(n)}).$$}
\end{Corollary}

In view of  this result, one may guess the equality $I^{(n)}=I^n:\mathfrak{m}^{\infty}$  holds for $n\geq 1$ if $I$ is a  two-dimensional squarefree monomial ideal. But this is not the case since $a_1(S/I_G^n)\neq a_1(S/I_G^{(n)})$ in general.

In order to obtain a formula for    $\reg(S/I_G^n)$,  we have to obtain the value for  $a_0(S/I_G^n)$, which  seems much more difficult. As a first step to compute $\reg(R/I_G^n)$, it would be of interest to  compute  the depth function $\mbox{depth}( S/I_G^n)$.

\vspace{2mm}

\noindent{\bf Acknowledgement}: We would like to thank the referees  for their valuable comments and advices, which  helped to improve the presentation of this manuscript. This project is supported by  NSFC (No. 11971338)


\begin{thebibliography}{}

\bibitem{BBH} A.~Banerjee, S.K.~Beyarslan, H.T.~H$\grave{\mathrm{a}}$, Regularity of powers of edge ideals: from local properties to global bounds, arXiv: 1805.01434v2

\bibitem{BHT} S.K. Beyarslan, H.T. H$\grave{a}$, T.N. Trung, Regularity of powers of forests and cycles, Journal of Algebraic Combinatorics, 42(2015), 1077-1095


\bibitem{BS} M.~Brodmann, R.Y.~Sharp, \textit{ Local cohomology: an algebraic introduction with geometric applications}. Second edition. Cambridge Studies in Advanced Mathematics, 136. Cambridge University Press, 2013




\bibitem{Ch} M.~Chardin, \textit{Regularity stabilization for the powers of graded M-primary ideals.} Proc. Amer.
Math. Soc. 143 (2015), no. 8, 3343-3349.

\bibitem{CH} A.~Conca, J.~Herzog, \textit{ Castelnuovo-Mumford regularity of products of ideals.} Collect. Math. 54 (2003), no. 2, 137-152.

\bibitem{CHT} S.D.~Cutkosky, J.~Herzog, N.V.~Trung, \textit{Asymptotic behaviour of the Castelnuovo-Mumford regularity},
Compositio Math. 118 (1999), 243-261.



\bibitem{EU} D.~Eisenbud, B.~Ulrich, \textit{Notes on regularity stabilization}, Proc. Amer. Math. Soc. 140 (2012), no. 4, 1221-1232.
\bibitem{HNTT} H.T.~H$\grave{a}$, H.D.~Nguyen, N.V.~Trung, T.N.~Trung, \textit{Symbolic powers of sums of ideals}, to appear in Math. Z.

\bibitem{HT} N.T.~Hang, T.N.~Trung, \textit{Regularity of powers of cover ideals of unimodular hypergraphs},  Journal of Algebra, 513(2018), 159-176.

\bibitem{HH}
J.~Herzog,  T.~Hibi, \textit{Monomial Ideals}, Graduate Text in Mathematics  260, Springer,  2011.

\bibitem{HRV}
J.~Herzog, A.~Rauf, M.~Vladoiu \textit{The stable set of associated prime ideals of a polymatroidal ideal}, Journal of Algebraic Combinatorics,  37 (2013), no. 2,  289-312.



\bibitem{HT2} L.T.~Hoa, T.N.~Trung, \textit{Castelnuovo-Mumford regularity of symbolic powers of two-dimensional
square-free monomial ideals}, J. Commut. Algebra 8(2016), no. 1, 77-88.








\bibitem{JNS} A.V.~Jayanthan, N.~Narayanan, S.~Selvaraja, \textit{Regularity of Powers of Bipartite Graphs,} Journal of Algebraic Combinatorics, 47(2018), 17-38.

\bibitem{K} V.~Kodiyalam,  \textit{Asymptotic behaviour of Castelnuovo-Mumford regularity}, Proc. A.M.S. 128(2000), 407-411.

\bibitem{MN1} N.C.~Minh, Y.~Nakamura, \textit{The Buchsbaum property of symbolic powers of Stanley-Reisner ideals of dimension 1},  J. Pure Appl. Algebra  215 (2011), 161-167.
\bibitem{MN2} N.C.~Minh, Y.~Nakamura, \textit{Buchsbaumness of ordinary powers of two-dimensional squarefree monomial ideals,} Journal of algebra, 327(2011), 292-306.



\bibitem{MT} N.C.~Minh, N.V.~Trung, \textit{Cohen-Macaulayness of powers of two-dimensional squarefree monomial ideals}, Journal of algebra, 322(2009), 4219-4227.

    \bibitem{NV}  H.D.~Nguyen,   T.~Vu,  \textit{Homological invariants of powers of fiber products},  Acta Math Vietnam (2019). https://doi.org/10.1007/s40306-018-00317-y


\bibitem{NT} N.~Terai, N.V.~Trung, \textit{Cohen-Macaulayness of large powers of Stanley-Reisner ideals}, Adv. Mathematics, 229(2012), 711-730



\bibitem{RTV} M.E.~Rossi, N.V.~Trung, G.~Valla, \textit{Castelnuovo-Mumford Regularity and finiteness of Hilbert Functions}, Lect. Notes Pure Appl. Math., 244, Chapman \& Hall/CRC, Boca Raton, FL, 2006.

\bibitem{T} Y.~Takayama, \textit{Combinatorial characterizations of generalized Cohen-Macaulay monomial ideals},
Bull. Math. Soc. Sci. Math. Roumanie (N.S.) 48(2005), 327-344.

\bibitem{TW}  N.V.~Trung, Hsin-Ju~Wang, \textit{On the asymptotic linearity of
Castelnuovo-Mumford regularity}, J. Pure Appl. Algebra 201(2005), 42-48.


\bibitem{W} D.B.~ West, \textit{Introduction to Graph Theory}, 2nd ed., Prentice-Hall, 2001

\end{thebibliography}
\end{document}